\documentclass[11pt]{article}

\usepackage{amssymb}
\usepackage{amsmath}
\usepackage{color}
\usepackage{rotating}
\usepackage{theorem}
\usepackage{tikz-cd}

\numberwithin{equation}{section}

\theorembodyfont{\slshape}

  \newtheorem{THM}{Theorem}[section]
  \newtheorem{LEM}[THM]{Lemma}
  \newtheorem{PROP}[THM]{Proposition}
  \newtheorem{COR}[THM]{Corollary}

\theorembodyfont{\rmfamily}

  \newtheorem{EX}{Example}[section] \let\oldendEX\endEX \def\endEX{$\diamondsuit$\oldendEX}
  \newtheorem{EXA}{Main example}[section] \let\oldendEXA\endEXA \def\endEXA{$\diamondsuit$\oldendEXA}
  \newtheorem{EXB}{Dual example}[section] \let\oldendEXB\endEXB \def\endEXB{$\diamondsuit$\oldendEXB}

\newif\ifQEDsign
\newcommand{\QED}{\global\QEDsigntrue\hfill$\square$}

\newenvironment{proof}%
    {\par\noindent\textit{Proof.}\global\QEDsignfalse}%
    {\ifQEDsign\else\QED\fi\par\bigskip\par}

\def\labelenumi{(\roman{enumi})}

\newcommand{\iso}{\mathrm{iso}}
\newcommand\toCC{\overset{\CC}\longrightarrow}
\newcommand\toCCC{\overset{\CC^*}\longrightarrow}
\newcommand\toAAA{\overset{\AAA^*}\longrightarrow}

\newcommand\toCCIII{\overset{\CC_1 \times \CC_2}\longrightarrow}

\renewcommand{\le}{\leqslant}
\renewcommand{\ge}{\geqslant}
\newcommand{\0}{\varnothing}

\renewcommand{\sec}{\cap}
\renewcommand{\phi}{\varphi}
\renewcommand{\epsilon}{\varepsilon}
\newcommand{\UNION}{\bigcup}
\newcommand{\AAA}{\mathbb{A}}

\newcommand{\CC}{\mathbb{C}}
\newcommand{\CCfin}{\mathbb{C}_{\mathit{fin}}}

\newcommand{\NN}{\mathbb{N}}
\newcommand{\QQ}{\mathbb{Q}}
\newcommand{\RR}{\mathbb{R}}

\newcommand{\union}{\cup}
\newcommand{\restr}[2]{\mbox{$#1$}\mbox{$\upharpoonright$}_{#2}}
\newcommand{\restrindex}[2]{\mbox{\scriptsize$#1$}\mbox{$\upharpoonright$}_{#2}}
\newcommand{\Boxed}[1]{\mbox{$#1$}}

\newcommand{\id}{\mathrm{id}}

\newcommand{\Ob}{\mathrm{Ob}}

\newcommand{\Sym}{\mathrm{Sym}}

\newcommand{\op}{\mathit{op}}

\newcommand{\diam}{\mathrm{diam}}

\newcommand{\calA}{\mathcal{A}}
\newcommand{\calB}{\mathcal{B}}
\newcommand{\calC}{\mathcal{C}}
\newcommand{\calD}{\mathcal{D}}
\newcommand{\calE}{\mathcal{E}}
\newcommand{\calF}{\mathcal{F}}

\newcommand{\calP}{\mathcal{P}}

\newcommand{\calS}{\mathcal{S}}

\newcommand{\Fraisse}{Fra\"\i ss\'e}

\DeclareMathOperator{\Aut}{Aut}

\DeclareMathOperator{\Age}{Age}
\DeclareMathOperator{\ProjAge}{ProjAge}

\title{The Kechris-Pestov-Todor\v cevi\'c correspondence\\from the point of view of category theory}
\author{Dragan Ma\v sulovi\'c\\
        Department of Mathematics and Informatics\\
        Faculty of Sciences, University of Novi Sad, Serbia\\
        email: dragan.masulovic@dmi.uns.ac.rs}

\begin{document}
\maketitle

\begin{abstract}
  The Kechris-Pestov-Todor\v cevi\'c correspondence (KPT-correspondence for short)
  is a surprising correspondence between model theory, combinatorics and topological dynamics.
  In this paper we present a categorical re-interpretation of (a part of) the KPT-correspondence with the aim of
  proving a dual statement.
  Our strategy is to take a ``direct'' result and then analyze the necessary infrastructure that makes the result
  true by providing a purely categorical proof of the categorical version of the result. We can then capitalize on the Duality Principle
  to obtain the dual statements almost for free.
  We believe that the dual version of the KPT-correspondence can not only provide the new insights into the interplay of
  combinatorial, model-theoretic and topological phenomena this correspondence binds together,
  but also explores the limits to which categorical treatment of combinatorial phenomena can take us.

  \bigskip


  \bigskip

  \noindent \textbf{AMS Subj.\ Classification (2020):} 18A05, 05C55, 37B99
\end{abstract}

\section{Introduction}

The Kechris-Pestov-Todor\v cevi\'c correspondence (KPT-correspondence for short)
is a surprising correspondence between model theory, combinatorics and topological dynamics
published in 2005 in~\cite{KPT}. For a locally finite countable homogeneous first-order structure $\calF$
the paper \cite{KPT} establishes a correspondence between combinatorial properties of $\Age(\calF)$,
the class of finite substructures of $\calF$, and dynamical properties of $\Aut(\calF)$ in the following sense:
$\Aut(\calF)$ is extremely amenable if and only if $\Age(\calF)$ has the Ramsey property.
In case $\Aut(\calF)$ is not extremely amenable \cite{KPT} offers a technique to compute its universal minimal
flow in case the structure $\calF$ can be expanded by a linear order~$<$ in a particular way (see \cite{KPT} for details).

The KPT-correspondence was later generalized to uncountable structures (see \cite{bartosova1,bartosova2}
and \cite{ferenzi-lopez-abad-mbombo-todorcevic} for a more recent result), while
a generalization from the model-theoretic point of view can be found in~\cite{krupinski-pillay}.
In this paper we present a categorical re-interpretation of (a part of) the KPT-correspondence with the aim of
proving a dual statement.
It follows that for a projectively locally finite projectively homogeneous object $F$ in a suitable category $\CC$,
the automorphism group of $F$ endowed with an appropriate topology is extremely amenable if and only if the
projective age of $F$ has the dual Ramsey property.

It was Leeb who pointed out already in 1970s \cite{Leeb} that the use of category theory can be quite helpful both in the
formulation and in the proofs of results pertaining to structural Ramsey theory. This point of view was formally supported only recently
\cite{masul-scow} when it was established that the Ramsey property (formulated in the language of category theory) is a genuine categorical property.
Effectively, it means that this purely combinatorial property can be studied by a completely new set of tools -- the tools of category theory.
One possible way to do so is to take a ``direct'' Ramsey result and then analyze the necessary infrastructure that makes the result
true by providing a purely categorical proof of the categorical version of the result. We can then capitalize on the Duality Principle
to obtain the dual statements almost for free.

The search for dual Ramsey statements has been an important research direction in the past 50 years because dual Ramsey results
are relatively rare in comparison to the vast number of "direct" Ramsey results.
We believe that the dual version of the KPT-correspondence can not only provide the new insights into the interplay of
combinatorial, model-theoretic and topological phenomena this correspondence binds together,
but also explores the limits to which categorical treatment of combinatorial phenomena can take us.

This paper builds on many interpretations of the KPT-correspondence, but most notably on papers by
Nguyen Van The \cite{vanthe-more,vanthe-finramdeg}, Zucker \cite{zucker} and Dasilva Barbosa \cite{dasilvabarbosa}.

In Section~\ref{akpt.sec.prelim} we briefly fix the basic notions and notation. The notions that do not come from
category were imported from Fra\"\i ss\'e theory (e.g.\ the age of a structure) and model theory (e.g.\ the notion of a
locally finite structure). The Fra\"\i ss\'e theory provides the tools to understand countable structures by looking at
their finite approximations. Hence, it is essential to understand what a finite substructure of a structure is.
This is trivial in case of first-order structures, of course, but becomes quite a challenge in the setting of category theory.
It turns out that it suffices to consider subobjects that are ``finite in nature'' in the following sense:
the hom-set $\hom(A, B)$ has to be finite whenever $A$ and $B$ are ``finite'', and each ``finite'' object has to have
only finitely many ``finite'' subobjects (up to isomorphism).

In Section \ref{akpt.sec.ram-degs} we introduce the Ramsey property and Ramsey degrees in two
settings: for objects and for morphisms. Structural Ramsey theory has traditionally been the study of the Ramsey property for
finite structures. This view shifted slowly to a new perspective (see \cite{GLR,P-V,zucker,mu-pon})
where it became more convenient to consider colorings of embeddings instead of coloring of substructures.
This point of view is essentially a categorical point of view as demonstrated in \cite{masul-scow}.
Not surprisingly, the dual Ramsey property, which has been studied from the very beginning of the structural Ramsey theory in
the early 1970s, turns out to be nothing but the categorical dual of the Ramsey property. This is the starting
point of the present paper.

In Section~\ref{akpt.sec.RPEA} we provide the abstract categorical proof of the KPT-correspondence and,
by the straightforward application of the Duality Principle, obtain the dual of KPT-correspondence.
Roughly, the result claims that under certain assumptions
the automorphism group of a projectively locally finite projectively homogeneous object 
is extremely amenable if and only if the projective age of the object has the dual Ramsey property.

Ramsey property imposes a very strong requirement of a class of finite structures, so it is not surprising
that many classes of finite structures (such as finite graphs and finite partially ordered sets) do not enjoy the property.
It is quite common, though, that after expanding the finite structures under consideration with appropriately chosen
linear orders or unary predicates the resulting class of expanded structures has the Ramsey property.
In Section~\ref{akpt.sec.exp-group-acts} we consider expansions as nothing but special forgetful functors (see \cite{dasilvabarbosa}).
We introduce several additional requirements in the spirit of \cite{KPT} that make expansion particularly tame and
enable us to consider group actions of a particular kind.
In Section~\ref{akpt.sec.ramsey-expansions} we consider the existence of expansions that have the Ramsey property in the spirit
of~\cite{vanthe-finramdeg}. We conclude the paper by showing that even in this very abstract setting
the existence of finite Ramsey degrees is a necessary and sufficient condition for the Ramsey expansions to exist.
The dual statement is then immediate: the existence of finite dual Ramsey degrees is a necessary and sufficient condition for
the expansions with the dual Ramsey property to exist.

\section{Preliminaries}
\label{akpt.sec.prelim}

Let us quickly fix some notation. Let $\CC$ be a category. By $\Ob(\CC)$ we denote the class of all the objects in $\CC$.
Hom-sets in $\CC$ will be denoted by $\hom_\CC(A, B)$, or simply $\hom(A, B)$ when $\CC$ is clear from the context.
The identity morphism will be denoted by $\id_A$ and the composition of morphisms by $\Boxed\cdot$ (dot).
If $\hom_\CC(A, B) \ne \0$ we write $A \toCC B$. Let $\iso_\CC(A, B)$ denote the set of all the
invertible morphisms $A \to B$, and let $\Aut_\CC(A) = \iso(A, A)$ denote the set of all the
\emph{automorphisms of $A$}. As usual, $\CC^\op$ denotes the opposite category.

All the categories in this paper are locally small. We shall explicitly state this assumption in the
formulation of main results, but will often omit the explicit statement of this fact in the formulation
of auxiliary statements.

\bigskip

Let $\CC$ be a locally small category and
let $\CCfin$ be a full subcategory of $\CC$ such that the following holds:
\begin{description}
\item[(C1)]
   all the morphisms in $\CC$ are mono;
\item[(C2)]
  $\Ob(\CCfin)$ is a set;
\item[(C3)]
  for all $A, B \in \Ob(\CCfin)$ the set $\hom(A, B)$ is finite;
\item[(C4)]
  for every $F \in \Ob(\CC)$ there is an $A \in \Ob(\CCfin)$ such that $A \toCC F$; and
\item[(C5)]
  for every $B \in \Ob(\CCfin)$ the set $\{A \in \Ob(\CCfin) : A \toCC B\}$ is finite.
\end{description}
We think of objects in $\CCfin$ as templates of finite objects in $\CC$.

For the remainder of the section let us fix a locally small category $\CC$ and its full subcategory $\CCfin$ which
satisfies (C1)--(C5). The notions we introduce below depend on the pair $(\CC, \CCfin)$ but we shall often
omit the context from the notation to keep it simple.

Let $\AAA$ be a full subcategory of $\CCfin$.
An object $F \in \Ob(\CC)$ is \emph{homogeneous for $\AAA$} if for every $A \in \Ob(\AAA)$ and every pair of morphisms
$e_1, e_2 \in \hom(A, F)$ there is a $g \in \Aut(F)$ such that $g \cdot e_1 = e_2$:
\begin{center}
  \begin{tikzcd}
    F \arrow[rr, "g"] & & F \\
     & A \arrow[ul, "e_1"] \arrow[ur, "e_2"'] &
  \end{tikzcd}
\end{center}
An object $F \in \Ob(\CC)$ is \emph{homogeneous in $\CC$ (with respect to $\CCfin$)} if it is homogeneous for $\CCfin$.
For historical reasons (see~\cite{irwin-solecki}) we say that $F \in \Ob(\CC)$ is
\emph{projectively homogeneous} if it is homogeneous in $\CC^\op$ (with respect to $\CCfin^\op$).

An object $F$ is \emph{universal for $\AAA$} if $A \toCC F$ for all $A \in \Ob(\AAA)$, and it is
\emph{universal in $\CC$ (with respect to $\CCfin$)} if it is universal for~$\CCfin$.
Dually, $F$ is \emph{projectively universal in $\CC$ (with respect to $\CCfin$)} if $F$ is universal
in $\CC^\op$ (with respect to $\CCfin^\op$).
Let us define the \emph{age of $F$ in $(\CC,\CCfin)$} by
$$
  \Age_{(\CC,\CCfin)}(F) = \{A \in \Ob(\CCfin) : A \toCC F \}
$$
and the \emph{projective age of $F$ in $(\CC,\CCfin)$} by
$$
  \ProjAge_{(\CC,\CCfin)}(F) = \Age_{(\CC^\op,\CCfin^\op)}(F).
$$
Clearly, every $F$ is universal for its age and is projectively universal for its projective age.
Whenever $\CC$ and $\CCfin$ are fixed we shall simply write $\Age(F)$ and $\ProjAge(F)$.

An $F \in \Ob(\CC)$ is \emph{locally finite for $\AAA$} if
\begin{itemize}
\item
  for every $A, B \in \Ob(\AAA)$ and every $e \in \hom(A, F)$, $f \in \hom(B, F)$ there exist $D \in \Ob(\AAA)$,
  $r \in \hom(D, F)$, $p \in \hom(A, D)$ and $q \in \hom(B, D)$ such that $r \cdot p = e$ and $r \cdot q = f$:
  \begin{center}
      \begin{tikzcd}
        D \arrow[rr, "r"] & & F & &  \\
        & A \arrow[ul, "p"] \arrow[ur, "e"' near start] & & B \arrow[ulll, "q" near start] \ar[ul, "f"']
      \end{tikzcd}
  \end{center}
\item
  and for every $H \in \Ob(\CC)$, $r' :\in \hom(H, F)$, $p' \in \hom(A, H)$ and $q' \in \hom(B, H)$
  such that $r' \cdot p' = e$ and $r' \cdot q' = f$ there is an $s \in \hom(D, H)$ such that the diagram below commutes
  \begin{center}
      \begin{tikzcd}
        D \arrow[rr, "r"] \arrow[rrrr, dotted, bend left=20, "s"] & & F & & H \arrow[ll, "r'"'] \\
        & A \arrow[ul, "p"] \arrow[ur, "e" near end] \arrow[urrr, "p'"' near end] & & B \arrow[ulll, "q"' near end] \arrow[ul, "f"' near end]
            \arrow[ur, "q'"']
      \end{tikzcd}
  \end{center}
\end{itemize}
An $F \in \Ob(\CC)$ is \emph{locally finite in $\CC$ (with respect to $\CCfin$)} if it is locally finite for $\CCfin$,
and it is \emph{projectively locally finite in $\CC$ (with respect to $\CCfin$)} if it is locally finite in $\CC^\op$.

\emph{The automorphisms of $F$ are finitely separated in $\CC$ (with respect to $\CCfin$)} if
the following holds for all $f, g \in \Aut(F)$ such that $f \ne g$:
there is an $A \in \Ob(\CCfin)$ and an $e \in \hom(A, F)$
such that $f \cdot e \ne g \cdot e$.
\emph{The automorphisms of $F$ are projectively finitely separated} if
the automorphisms of $F$ are finitely separated in $\CC^\op$.

Finally, a category $\CC$ is \emph{directed} if for all $A, B \in \Ob(\CC)$ there exists a $C \in \Ob(\CC)$
such that $A \toCC C$ and $B \toCC C$.
A category $\CC$ is \emph{projectively directed} if $\CC^\op$ is directed.

\section{Ramsey property and Ramsey degrees in a category}
\label{akpt.sec.ram-degs}

Let us start by introducing the Ramsey property and Ramsey degrees in two
settings: for objects and for morphisms.

For $k \in \NN$, a $k$-coloring of a set $S$ is any mapping $\chi : S \to k$, where,
as usual, we identify $k$ with $\{0, 1,\ldots, k-1\}$.

For integers $k \ge 2$ and $t \ge 1$, and objects $A, B, C \in \Ob(\CC)$
such that $A \toCC B$ we write
$
  C \longrightarrow (B)^{A}_{k, t}
$
to denote that for every $k$-coloring $\chi : \hom(A, C) \to k$ there is a morphism
$w \in \hom(B, C)$ such that $|\chi(w \cdot \hom(A, B))| \le t$.
(For a set of morphisms $F$ we let $w \cdot F = \{ w \cdot f : f \in F \}$.)

In case $t = 1$ we write
$
  C \longrightarrow (B)^{A}_{k}.
$
On the other hand, we write $C \longrightarrow (B)^{A}_{\Boxed< \omega, t}$ to denote that
$C \longrightarrow (B)^{A}_{k, t}$ for all $k \ge 2$.

A category $\CC$ has the \emph{Ramsey property} if
for every integer $k \ge 2$ and all $A, B \in \Ob(\CC)$ there is a
$C \in \Ob(\CC)$ such that $C \longrightarrow (B)^{A}_k$.
A category $\CC$ has the \emph{dual Ramsey property} if $\CC^\op$ has the Ramsey property.

For $A \in \Ob(\CC)$ let $t_\CC(A)$ denote the least positive integer $n$ such that
for all $k \ge 2$ and all $B \in \Ob(\CC)$ there exists a $C \in \Ob(\CC)$ such that
$C \longrightarrow (B)^{A}_{k, n}$, if such an integer exists.
Otherwise put $t_\CC(A) = \infty$.

A category $\CC$ has the \emph{finite Ramsey degrees} if $t_\CC(A) < \infty$
for all $A \in \Ob(\CC)$.
By a dual procedure we can straightforwardly introduce the notion of dual Ramsey degrees.
We then say that a category $\CC$ has the \emph{finite dual Ramsey degrees} if $\CC^\op$ has finite Ramsey degrees.

Define $\sim_A$ on $\hom(A, B)$ as follows:
for $f, f' \in \hom(A, B)$ we let $f \sim_A f'$ if $f' = f \cdot \alpha$
for some $\alpha \in \Aut(A)$. Then
$$
  \binom B A = \hom(A, B) / \Boxed{\sim_A}
$$
corresponds to all subobjects of $B$ isomorphic to $A$.
For an integer $k \ge 2$ and $A, B, C \in \Ob(\CC)$ we write
$$
  C \overset\sim\longrightarrow (B)^{A}_{k, t}
$$
to denote that for every $k$-coloring
$
  \chi : \binom CA \to k
$
there is a morphism $w : B \to C$ such that $|\chi(w \cdot \binom BA)| \le t$.
(Note that $w \cdot (f / \Boxed{\sim_A}) = (w \cdot f) / \Boxed{\sim_A}$ for $f / \Boxed{\sim_A} \in \binom B A$.)
Instead of $C \overset\sim\longrightarrow (B)^{A}_{k, 1}$ we simply write
$C \overset\sim\longrightarrow (B)^{A}_{k}$.

A category $\CC$ has the \emph{Ramsey property for objects} if
for every integer $k \ge 2$ and all $A, B \in \Ob(\CC)$ there is a
$C \in \Ob(\CC)$ such that $C \overset\sim\longrightarrow (B)^{A}_k$.

For $A \in \Ob(\CC)$ let $t^\sim_\CC(A)$ denote the least positive integer $n$ such that
for all $k \ge 2$ and all $B \in \Ob(\CC)$ there exists a $C \in \Ob(\CC)$ such that
$C \overset\sim\longrightarrow (B)^{A}_{k, n}$, if such an integer exists.
Otherwise put $t^\sim_\CC(A) = \infty$.

\begin{PROP}\label{rdbas.prop.sml}
  Let $\CC$ be a locally small category such that all the morphisms in $\CC$ are mono
  and let $A \in \Ob(\CC)$. Then $t_\CC(A)$ is finite if and only if both $t^\sim_\CC(A)$ and $\Aut(A)$ are finite,
  and in that case
  $$
    t_\CC(A) = |\Aut(A)| \cdot t^\sim_\CC(A).
  $$
\end{PROP}
\begin{proof}
  Assume, first, that $|\Aut(A)| = \infty$.
  Let us show that $t_\CC(A) = \infty$ by showing that
  $t_\CC(A) \ge n$ for every $n \in \NN$. Fix an $n \in \NN$ and $X \subseteq \Aut(A)$ such that $|X| = n$.
  Since $t^\sim_\CC(A) \ge 1$ there is
  a $k \ge 2$ and a $B \in \Ob(\CC)$ such that for every $C \in \Ob(\CC)$ one can find a coloring
  $\chi : \binom CA \to k$ such that for every $w : B \to C$ we have that
  $|\chi(w \cdot \binom BA)| \ge 1$. This is, of course, trivial. We need this argument
  just to ensure the existence of a $B$ such that $A \toCC B$.
  
  Let $\binom CA = \hom(A, C) / \Boxed{\sim_A}  = \{H_i : i \in I\}$ for some index set $I$.
  For each $i \in I$ choose a representative $h_i \in H_i$. Then $H_i = h_i \cdot \Aut(A)$.
  Fix an arbitrary $\xi \in X$ and define $\chi' : \hom(A, C) \to X$ as follows:
  \begin{itemize}
  \item[]
    if $g = h_i \cdot \alpha$ for some $i \in I$ and some $\alpha \in X$ then $\chi'(g) = \alpha$;
  \item[]
    otherwise $\chi'(g) = \xi$.
  \end{itemize}
  Take any $w : B \to C$. Let $f \in \hom(A, B)$ be arbitrary. Then:
  $$
    |\chi'(w \cdot \hom(A, B))| \ge |\chi'(w \cdot f \cdot \Aut(A))|.
  $$
  Clearly, $w \cdot f \cdot \Aut(A) = h_i \cdot \Aut(A)$ for some $i \in I$, so
  $$
    |\chi'(w \cdot \hom(A, B))| \ge |\chi'(h_i \cdot \Aut(A))| = n.
  $$
  This completes the proof in case $\Aut(A)$ is infinite.
  
  Let us now move on to the case when $\Aut(A)$ is finite.

  Let $t^\sim_\CC(A) = n$ for some $n \in \NN$.
  Take any $k \ge 2$ and any $B \in \Ob(\CC)$. Then there is a $C \in \Ob(\CC)$
  such that $C \overset{\sim}\longrightarrow (B)^{A}_{2^k, n}$. Let $\chi : \hom(A, C) \to k$ be an
  arbitrary coloring. Construct $\chi' : \binom CA \to \calP(k)$ as follows:
  $$
    \chi'(f / \Boxed{\sim_A}) = \chi(f / \Boxed{\sim_A})
  $$
  (here, $\chi$ is applied to a set of morphisms to produce a set of colors, which is an element of $\calP(k)$).
  Then $C \overset{\sim}\longrightarrow (B)^{A}_{2^k, n}$ implies that there exists a $w : B \to C$ such that
  $|\chi'(w \cdot \binom BA)| \le n$. Since $w \cdot (f / \Boxed{\sim_A}) = (w \cdot f) / \Boxed{\sim_A}$ it follows
  that
  $$
    w \cdot \binom BA = \{ (w \cdot f) / \Boxed{\sim_A} : f \in \hom(A, B) \}.
  $$
  Moreover, the morphisms in $\CC$ are mono, so $|u / \Boxed{\sim_A}| = |\Aut(A)|$
  for each morphism $u \in \hom(A, C)$. Therefore,
  $|\chi'(w \cdot \binom BA)| \le n$ implies that $|\chi(w \cdot \hom(A, B))| \le n \cdot |\Aut(A)|$
  proving that $t_\CC(A) \le n \cdot |\Aut(A)| = t^\sim_\CC(A) \cdot |\Aut(A)|$.
  
  For the other inequality note that $t^\sim_\CC(A) = n$ also implies that there is a $k \ge 2$ and a
  $B \in \Ob(\CC)$ such that for every $C \in \Ob(\CC)$ one can find a coloring
  $\chi_C : \binom CA \to k$ with the property that for every $w \in \hom(B, C)$ we have that
  $|\chi_C(w \cdot \binom BA)| \ge n$. Let $\ell = k \cdot |\Aut(A)|$ and take an arbitrary $C \in \Ob(\CC)$.
  Let $\binom CA = \hom(A, C) / \Boxed{\sim_A}  = \{H_i : i \in I\}$ for some index set $I$.
  For each $i \in I$ choose a representative $h_i \in H_i$. Then $H_i = h_i \cdot \Aut(A)$.  
  Since all the morphisms in $\CC$ are mono, for each $f \in \hom(A, C)$ there is a unique
  $i \in I$ and a unique $\alpha \in \Aut(A)$ such that $f = h_i \cdot \alpha$. Let us denote this
  $\alpha$ by $\alpha(f)$. Consider the following coloring:
  $$
    \xi : \hom(A, C) \to k \times \Aut(A) : f \mapsto (\chi_C(f/\Boxed{\sim_A}), \alpha(f))
  $$
  and take any $w \in \hom(B, C)$. Since
  $|\chi_C(w \cdot \binom BA)| \ge n$, it easily follows that $|\xi(w \cdot \hom(A, B))| \ge n \cdot |\Aut(A)|$
  proving that $t_\CC(A) \ge n \cdot |\Aut(A)| = t^\sim_\CC(A) \cdot |\Aut(A)|$.

  Assume, finally, that $t^\sim_\CC(A) = \infty$ and let us show that $t_\CC(A) = \infty$ by showing that
  $t_\CC(A) \ge n$ for every $n \in \NN$. Fix an $n \in \NN$. Since $t_\CC(A) = \infty$, there is
  a $k \ge 2$ and a $B \in \Ob(\CC)$ such that for every $C \in \Ob(\CC)$ one can find a coloring
  $\chi : \binom CA \to k$ such that for every $w : B \to C$ we have that
  $|\chi(w \cdot \binom BA)| \ge n$. Then the coloring $\chi' : \hom(A, C) \to k$ defined by
  $$
    \chi'(f) = \chi(f / \Boxed{\sim_A})
  $$
  has the property that $|\chi(w \cdot \hom(A, B))| \ge n$.

  This completes the proof.
\end{proof}

As an immediate corollary we have the following:

\begin{COR}
  Let $\CC$ be a locally small category such that all the morphisms in $\CC$ are mono
  and let $A \in \Ob(\CC)$. Then

  $(a)$ $t_\CC(A) \ge |\Aut(A)|$;
  
  $(b)$ if $t_\CC(A) \le n$ then $|\Aut(A)| \le n$;
  
  $(c)$ if $t_\CC(A) = 1$ then $A$ is rigid.
\end{COR}
\begin{proof}
  $(a)$ follows from Proposition~\ref{rdbas.prop.sml}, while
  $(b)$ and $(c)$ are direct consequences of $(a)$.
\end{proof}

Let us also show that Ramsey degrees are multiplicative.

\begin{THM}
  Let $\CC_1$ and $\CC_2$ be categories whose morphisms are mono.
  Assume that for each $i \in \{1, 2\}$ and all $A_i, B_i \in \Ob(\CC_i)$ we have that
  $\hom_{\CC_i}(A_i, B_i)$ is finite and that $t_{\CC_i}(A_i)$ is finite.
  Then for all $A_1 \in \Ob(\CC_1)$ and $A_2 \in \Ob(\CC_2)$ we have that
  $$
    t_{\CC_1 \times \CC_2}(A_1, A_2) \le t_{\CC_1}(A_1) \cdot t_{\CC_2}(A_2).
  $$
  Consequently,
  $$
    t^\sim_{\CC_1 \times \CC_2}(A_1, A_2) \le t^\sim_{\CC_1}(A_1) \cdot t^\sim_{\CC_2}(A_2).
  $$
\end{THM}
\begin{proof}
  The second part of the statement is an immediate consequence of the first part of the statement,
  Proposition~\ref{rdbas.prop.sml} and the fact that
  $$
    |\Aut_{\CC_1 \times \CC_2}(A_1, A_2)| = |\Aut_{\CC_1}(A_1)| \cdot |\Aut_{\CC_2}(A_2)|.
  $$
  To show the first part of the statement
  take any $k \ge 2$ and any $(B_1, B_2) \in \Ob(\CC_1 \times \CC_2)$ such that
  $(A_1, A_2) \toCCIII (B_1, B_2)$. Let $t_{\CC_1}(A_1) = n_1$ and
  $t_{\CC_2}(A_2) = n_2$, and choose $C_1 \in \Ob(\CC_1)$ and $C_2 \in \Ob(\CC_2)$ so that
  $$
    C_1 \longrightarrow (B_1)^{A_1}_{2^k, \; n_1} \text{\quad and\quad}
    C_2 \longrightarrow (B_2)^{A_2}_{k^h, \; n_2},
  $$
  where $h = |\hom_{\CC_1}(A_1, C_1)|$. Let us show that
  $$
    (C_1, C_2) \longrightarrow (A_1, A_2)^{(B_1, B_2)}_{k, \; n_1 \cdot n_2}.
  $$
  Take any coloring
  $
    \chi : \hom_{\CC_1 \times \CC_2}((A_1, A_2), (C_1, C_2)) \to k
  $ and let
  $$
    \chi' : \hom_{\CC_2}(A_2, C_2) \to k^{\hom_{\CC_1}(A_1, C_1)}
  $$
  be the coloring defined by $\chi'(e_2) = f_{e_2}$ where $f_{e_2}(e_1) = \chi(e_1, e_2)$.
  Since $C_2 \longrightarrow (B_2)^{A_2}_{k^h, \; n_2}$ there is a $w_2 \in \hom_{\CC_2}(B_2, C_2)$ such that
  \begin{equation}\label{sbrd.eq.t-prod-2}
    |\chi'(w_2 \cdot \hom_{\CC_2}(A_2, B_2))| \le n_2.
  \end{equation}
  Now define $\chi'' : \hom_{\CC_1}(A_1, C_1) \to \calP(k)$ by
  $$
    \chi''(e_1) = \{\chi(e_1, e_2) : e_2 \in w_2 \cdot \hom_{\CC_2}(B_2, C_2)\}.
  $$
  Since $C_1 \longrightarrow (B_1)^{A_1}_{2^k, \; n_1}$ there is a $w_1 \in \hom_{\CC_1}(B_1, C_1)$ such that
  \begin{equation}\label{sbrd.eq.t-prod-1}
    |\chi''(w_1 \cdot \hom_{\CC_1}(A_1, B_1))| \le n_1.
  \end{equation}
  Clearly, $(w_1, w_2) : (B_1, B_2) \to (C_1, C_2)$ so let us show that
  $$
    |\chi((w_1, w_2) \cdot \hom_{\CC_1 \times \CC_2}((A_1, A_2), (B_1, B_2)))| \le n_1 \cdot n_2.
  $$
  To start with, note that
  \begin{align*}
    \chi(&(w_1, w_2) \cdot \hom_{\CC_1 \times \CC_2}((A_1, A_2), (B_1, B_2))) =\\
      &= \{\chi(e_1, e_2) : e_1 \in w_1 \cdot \hom_{\CC_1}(A_1, B_1), e_2 \in w_2 \cdot \hom_{\CC_2}(A_2, B_2) \}\\
      &= \bigcup \{ \chi''(e_1) : e_1 \in w_1 \cdot \hom_{\CC_1}(A_1, B_1)\}.
  \end{align*}
  This union has at most $n_1$ distinct elements because of~\eqref{sbrd.eq.t-prod-1}.
  Now, fix an $e_1 \in w_1 \cdot \hom_{\CC_1}(A_1, B_1)$ and let us estimate the size of $\chi''(e_1)$:
  \begin{align*}
    \chi''(e_1)
      &= \{\chi(e_1, e_2) : e_2 \in w_2 \cdot \hom_{\CC_2}(B_2, C_2)\}\\
      &= \{f_{e_2}(e_1) : e_2 \in w_2 \cdot \hom_{\CC_2}(B_2, C_2)\}.
  \end{align*}
  Because of~\eqref{sbrd.eq.t-prod-2} we have that
  $$
    |\{f_{e_2} : e_2 \in w_2 \cdot \hom_{\CC_2}(B_2, C_2) \}| \le n_2,
  $$
  so these $n_2$ functions applied to a single value $e_1$ can produce at most $n_2$ distinct values. Therefore,
  $|\chi''(e_1)| \le n_2$ for each $e_1 \in w_1 \cdot \hom_{\CC_1}(A_1, B_1)$.
  Therefore, the union consists of at most $n_1$ distinct sets, and each
  set appearing in the union has at most $n_2$ elements, whence
  $$
    \chi((w_1, w_2) \cdot \hom_{\CC_1 \times \CC_2}((A_1, A_2), (B_1, B_2))) \le n_1 \cdot n_2.
  $$
  This completes the proof.
\end{proof}

Let $A$ and $B$ be sets and let $f : A \to B$ be a function. Recall that
\emph{the kernel of $f$} is the equivalence relation
$\ker f = \{(x, y) \in A^2 : f(x) = f(y)\}$.

Fix an object $F \in \Ob(\CC)$, and let $A, B \in \Ob(\CCfin)$ be such that $A \toCC B \toCC F$.
Every coloring $\chi : \hom(A, F) \to k$ and every $w \in \hom(B, F)$ induce a coloring
$\chi^{(w)} : \hom(A, B) \to k$ by $\chi^{(w)}(f) = \chi(w \cdot f)$.
Given $F$, a coloring $\lambda : \hom(A, B) \to t$, $t \ge 2$, is \emph{essential at $B$} if
for every coloring $\chi : \hom(A, F) \to k$ there is a $w \in \hom(B, F)$
such that $\ker \lambda \subseteq \ker \chi^{(w)}$.

\begin{LEM}\label{akpt.lem.ramseyF}
  Let $\AAA$ be a full subcategory of $\CCfin$ and fix an $F \in \Ob(\CC)$ which is universal and locally finite
  for $\AAA$. The following are equivalent for all $t \ge 2$ and all $A \in \Ob(\AAA)$:
  \begin{enumerate}\def\labelenumi{(\arabic{enumi})}
  \item
    $t_{\AAA}(A) \le t$;
  \item
    $F \longrightarrow (B)^A_{\Boxed< \omega, t}$ for all $B \in \Ob(\AAA)$ such that $A \toCC B$;
  \item
    for all $B \in \Ob(\AAA)$ such that $A \toCC B$
    there is a coloring $\lambda : \hom(A, B) \to t$ which is essential at~$B$.
  \end{enumerate}
\end{LEM}
\begin{proof}
  $(1) \Rightarrow (2)$:
  Take any coloring $\chi : \hom(A, F) \to k$, take any $B \in \Ob(\AAA)$ and find a
  $C \in \Ob(\AAA)$ such that $C \longrightarrow (B)^{A}_{k, t}$.
  Let $p : C \to F$ be any morphism (which exists
  because $F$ is universal for $\AAA$). Let $\xi : \hom(A, C) \to k$ be the coloring defined by
  $\xi(f) = \chi(p \cdot f)$. Then there is a $w : B \to C$
  such that $|\xi(w \cdot \hom(A, B))| \le t$. But $\xi(w \cdot \hom(A, B)) =
  \chi(p \cdot w \cdot \hom(A, B))$. Therefore, for $w' = p \cdot w : B \to S$ we have
  $|\chi(w' \cdot \hom(A, B))| \le t$.

  $(2) \Rightarrow (1)$:
  Fix a $B \in \Ob(\AAA)$ and assume that
  for every $C \in \Ob(\AAA)$ there is a coloring $\chi_C : \hom(A, C) \to k$
  with the property that for every $w : B \to C$ we have $|\chi_C (w \cdot \hom(A, B))| > t$.
  Let $X = k^{\hom(A, F)}$. Then $X$ is a compact Hausdorff space with respect to the Tychonoff topology
  (with $k$ discrete). For $C \in \Ob(\AAA)$ and $e : C \to F$ let
  $$
    \Phi_{C, e} = \{ \chi \in X : (\forall w : B \to C) \; |\chi(e \cdot w \cdot \hom(A, B))| > t\}.
  $$

  Let us show that $\Phi_{C, e}$ is a nonempty closed subset of $X$ for every $C \in \Ob(\AAA)$ and every $e : C \to F$.
  Take any $C \in \Ob(\AAA)$ and $e : C \to F$. Since $e \cdot \hom(A, C) \subseteq \hom(A, F)$,
  define $\chi : \hom(A, F) \to k$ by $\chi(e \cdot f) = \chi_C(f)$ for $f \in \hom(A, C)$
  and $\chi(g) = 0$ otherwise. Note that $\chi$ is well defined because $e$ is mono.
  Then for every $w : B \to C$ we have
  $\chi(e \cdot w \cdot \hom(A, B)) = \chi_C(w \cdot \hom(A, B))$.
  So, $|\chi(e \cdot w \cdot \hom(A, B))| = |\chi_C(w \cdot \hom(A, B))| > t$,
  whence $\chi \in \Phi_{C,e}$ and $\Phi_{C,e}$ is nonempty. To show that $\Phi_{C,e}$ is closed note that
  $$
    X \setminus \Phi_{C,e} = \{ \chi \in X : (\exists w : B \to C) \; |\chi(e \cdot w \cdot \hom(A, B))| \le t\}.
  $$
  Take any $\chi^* \in X \setminus \Phi_{C,e}$ and choose a $w^* : B \to C$ such that
  $\chi^*(e \cdot w^* \cdot \hom(A, B)) = \{c^*_0, \ldots, c^*_{s-1}\}$
  for some $s \le t$. Since $A, B \in \Ob(\AAA)$ we have that $\hom(A, B)$ is finite.
  Let $\hom(A, B) = \{f_0, \ldots, f_{n-1}\}$ and
  $$
    U = \{ \xi \in X : (\forall i < n)(\exists j < s) \xi(e \cdot w^* \cdot f_i) = c^*_j\}.
  $$
  Then $U$ is an open set in $X$ and $\chi^* \in U \subseteq X \setminus \Phi_e$. This completes the proof that
  $\Phi_{C,e}$ is a nonempty closed subset of $X$.

  \medskip

  For $B \in \Ob(\AAA)$ which we have fixed above put $\Phi_* = \{\Phi_{B,e} : e \in \hom(B, F)\}$
  and let us show that $\bigcap \Phi_* \ne \0$. Since $\Phi_*$ is a family of closed subsets of a compact Hausdorff space,
  it suffices to show that $\Phi_*$ has the finite intersection property.

  Take any $\Phi_{B,e_0}, \ldots, \Phi_{B,e_{n-1}} \in \Phi_*$. Because $F$ is locally finite
  there exist $D \in \Ob(\AAA)$, $p : D \to F$ and $m_i : B \to D$ such that
  $p \cdot m_i = e_i$, $i < n$.
  Let us show that $\Phi_{D,p} \subseteq \Phi_{B,e_0} \sec \ldots \sec \Phi_{B,e_{n-1}}$.
  Take any $\chi \in \Phi_{D,p}$ and fix an $i < n$.
  Then $|\chi(p \cdot w \cdot \hom(A, B))| > t$ for all $w : B \to D$.
  According to the assumption, for an arbitrary $w' : B \to B$, say $w' = \id_B$, we have
  $|\chi(p \cdot (m_i \cdot w') \cdot \hom(A, B))| > t$. But $p \cdot m_i = e_i$, so
  $|\chi(e_i \cdot w' \cdot \hom(A, B))| > t$, whence $\chi \in \Phi_{B,e_i}$. Thus
  $\0 \ne \Phi_{D,p} \subseteq \Phi_{B,e_i}$, for all~$i < n$.

  Therefore, $\Phi_*$ has the finite intersection property, so $\bigcap \Phi_* \ne \0$. Take any $\chi_0 \in \bigcap \Phi_*$.
  Then for every $e : B \to F$ we have $\chi_0 \in \Phi_e$, so in particular for $w = \id_B$ we have
  $\chi_0(e \cdot \hom(A, B)) > t$.

  $(3) \Rightarrow (2)$: Let $k \ge 2$ and $\chi : \hom(A, F) \to k$ be arbitrary.
  Since $\lambda$ is essential at $B$, there is a a $w \in \hom(B, F)$
  such that $\ker \lambda \subseteq \ker \chi^{(w)}$. Hence,
  $
    |\chi(w \cdot \hom(A, B))| \le t
  $.
  
  $(2) \Rightarrow (3)$: Aiming for a contradiction,
  suppose that (2) holds but that there is a $B \in \Ob(\AAA)$ with $A \toCC B$ such that no
  coloring in $t^{\hom(A, B)}$ is essential at $B$. So, for every $\lambda \in t^{\hom(A, B)}$ there exist
  $k_\lambda \ge 2$ and $\chi_\lambda : \hom(A, F) \to k_\lambda$ such that
  \begin{equation}\label{akpt.eq.IFF}
    \begin{array}{r@{}l}
    (\forall w \in \hom(B, F))&(\exists f, g \in \hom(A, B))\\
                              &(\lambda(f) = \lambda(g) \land \chi_\lambda(w \cdot f) \ne \chi_\lambda(w \cdot g)).
    \end{array}
  \end{equation}
  Note that $\hom(A, B)$ is a finite set due to $(C3)$, so let $t^{\hom(A, F)} = \{\lambda_0, \ldots, \lambda_{n-1}\}$.
  Consider the coloring
  $$
    \chi^* : \hom(A, F) \to \prod_{i < n} k_{\lambda_i}
  $$
  given by
  $$
    \chi^*(h) = (\chi_{\lambda_0}(h), \ldots, \chi_{\lambda_{n-1}}(h)).
  $$
  Because of (2) there is a $w \in \hom(B, F)$ such that $|\chi^*(w \cdot \hom(A, B))| \le t$.
  Let $\chi^*(w \cdot \hom(A, B)) = \{c_0, \ldots, c_{s-1}\}$ for some $s \le t$.
  Now, define $\beta : \hom(A, B) \to t$ by
  $$
    \beta(f) = i \text{ iff } \chi^*(w \cdot f) = c_i.
  $$
  Then $\beta \in t^{\hom(A, F)}$ so $\beta = \lambda_i$ for some $i < n$. By \eqref{akpt.eq.IFF}
  there exist $f, g \in \hom(A, B)$ such that
  $$
    \lambda_i(f) = \lambda_i(g) \text{ and } \chi_{\lambda_i}(w \cdot f) \ne \chi_{\lambda_i}(w \cdot g).
  $$
  From $\lambda_i(f) = \lambda_i(g)$ (that is, $\beta(f) = \beta(g)$) it follows that
  $\chi^*(w \cdot f) = \chi^*(w \cdot g)$ whence, by projecting onto the $i$th coordinate, we get
  $\chi_{\lambda_i}(w \cdot f) = \chi_{\lambda_i}(w \cdot g)$. Contradiction.
\end{proof}

Let $\AAA$ be a full subcategory of $\CCfin$,
let $F \in \Ob(\CC)$ be universal for $\AAA$ and let $A \in \Ob(\AAA)$ be arbitrary.
A coloring $\gamma : \hom(A, F) \to t$, $t \ge 2$, is \emph{essential} if
for every $B \in \Ob(\AAA)$ such that $A \toCC B$ and every $w \in \hom(B, F)$
the coloring $\gamma^{(w)} : \hom(A, B) \to t$ is essential at~$B$.

\begin{LEM}
  Let $\AAA$ be a full subcategory of $\CCfin$ and
  let $F \in \Ob(\CC)$ be universal and locally finite for $\AAA$.
  Let $A \in \Ob(\AAA)$ and $t \ge 2$ be arbitrary.
  Assume that for every $B \in \Ob(\AAA)$ such that $A \toCC B$
  there is a coloring $\lambda_B : \hom(A, B) \to t$ essential at~$B$.
  Then there exists an essential coloring $\gamma : \hom(A, F) \to t$.
\end{LEM}
\begin{proof}
  The proof follows by a typical compactness argument.
  Let $X = t^{\hom(A, F)}$. (Note that $X$ is a set because of the assumption that all the
  categories we work with are locally small.)
  Then $X$ is a compact Hausdorff space with respect to the Tychonoff topology
  (with $t = \{0, 1, \ldots, t-1\}$ taken as a discrete space).
  For $B \in \Ob(\AAA)$ such that $A \toCC B$ and for $e \in \hom(B, F)$ let
  \begin{align*}
    \Phi_{e, B} = \{ \beta \in X : \; & (\forall k \ge 2)(\forall \chi : \hom(A, F) \to k)\\
                                   & (\exists w \in \hom(B, F)) \ker \beta^{(e)} \subseteq \ker \chi^{(w)}  \}.
  \end{align*}

  To see that $\Phi_{e, B} \ne \0$ given $B$ and $e$ consider $\beta_{e, B} : \hom(A, F) \to t$ defined as follows:
  $$
    \beta_{e, B}(h) = \begin{cases}
      \lambda_B(f), & \text{if } h = e \cdot f \text{ for } f \in \hom(A, B),\\
      0, & \text{otherwise.}
    \end{cases}
  $$
  Then $\beta_{e, B}$ is well defined (because $e$ is mono) and $\beta_{e, B} \in \Phi_{e, B}$ because
  $\lambda_B$ is essential at~$B$ and $\beta_{e, B}^{(e)} = \lambda_B$.

  Next, let us show that each $\Phi_{e, B}$ is closed in $X$. Given $B$ and $e$, take any
  $\beta_0 \in X \setminus \Phi_{e, B}$. Then there exist $k_0 \ge 2$ and $\chi_0 : \hom(A, F) \to k_0$ such that
  \begin{align*}
    (\forall w \in \hom(B, F))&(\exists f, g \in \hom(A, B))\\
                              &(\beta_0^{(e)}(f) = \beta_0^{(e)}(g) \land \chi_0^{(w)}(f) \ne \chi_0^{(w)}(g)).
  \end{align*}
  Recall that $\hom(A, B)$ is finite because $A, B \in \Ob(\CCfin)$ and because of~$(C3)$. Let
  $\hom(A, B) = \{f_1, \ldots, f_n\}$, let $\beta_0(e \cdot f_1) = c_1$, \ldots, $\beta_0(e \cdot f_n) = c_n$
  and let
  $$
    O = \{ \beta \in t^{\hom(A, F)} : \beta(e \cdot f_i) = c_i, 1 \le i \le n \}.
  $$
  Then $O$ is open in $X$ and $\beta_0 \in O \subseteq X \setminus \Phi_{e,B}$. (Note that $\chi_0$ witnesses the
  fact that every $\beta \in O$ belongs to $X \setminus \Phi_{e,B}$.)

  Let $\Phi_* = \{\Phi_{e, B} : B \in \Ob(\AAA), A \toCC B, e \in \hom(B, F)\}$
  and let us show that $\bigcap \Phi_* \ne \0$. Since $\Phi_*$ is a family of closed subsets of a compact Hausdorff space,
  it suffices to show that $\Phi_*$ has the finite intersection property.
  Take any $\Phi_{e_0, B_0}, \ldots, \Phi_{e_{n-1}, B_{n-1}} \in \Phi_*$.
  Because $F$ is locally finite for $\AAA$ there exist $D \in \Ob(\AAA)$, $p \in \hom(D, F)$ and $m_i \in \hom(B_i, D)$ such that
  $p \cdot m_i = e_i$, $i < n$.
  
  Let us show that $\Phi_{p, D} \subseteq \Phi_{e_0, B_0} \sec \ldots \sec \Phi_{e_{n-1}, B_{n-1}}$.
  Take any $\beta \in \Phi_{p, D}$ and fix an $i < n$. To show that $\beta \in \Phi_{e_i, B_i}$ 
  let $k \ge 2$ and $\chi : \hom(A, F) \to k$ be arbitrary. Since $\beta \in \Phi_{p, D}$ there is a
  $w \in \hom(D, F)$ such that $\ker \beta^{(p)} \subseteq \ker \chi^{(w)}$, or, equivalently,
  \begin{equation}\label{akpt.eq.FIP}
    (\forall f', g' \in \hom(A, D)) (\beta(p \cdot f') = \beta(p \cdot g') \Rightarrow \chi(w \cdot f') = \chi(w \cdot g')).
  \end{equation}
  Take any $f, g \in \hom(A, B_i)$ and assume that $\beta(e_i \cdot f) = \beta(e_i \cdot g)$. Then
  $\beta(p \cdot m_i \cdot f) = \beta(p \cdot m_i \cdot g)$ whence
  $\chi(w \cdot m_i \cdot f) = \chi(w \cdot m_i \cdot g)$ by~\eqref{akpt.eq.FIP}.
  Therefore, we have shown that $\ker \beta^{(e_i)} \subseteq \ker \chi^{(w \cdot m_i)}$ which completes the proof
  of $\Phi_{p, D} \subseteq \Phi_{e_i, B_i}$ for an arbitrary $i < n$.

  Therefore, $\Phi_*$ has the finite intersection property so $\bigcap \Phi_* \ne \0$.
  Take any $\gamma \in \bigcap \Phi_*$. Then for every $B \in \Ob(\CCfin)$ such that $A \toCC B$
  and every $e \in \hom(B, F)$ we have that $\gamma \in \Phi_{e, B}$. But $\gamma \in \Phi_{e, B}$ if and
  only if $\gamma^{(e)}$ is essential at~$B$. Therefore, $\gamma$ is essential.
\end{proof}

As an immediate consequence of the above two lemmas we have the following.

\begin{PROP}\label{akpt.prop.glob-fin}
  Let $\CC$ be a locally small category and
  let $\CCfin$ be a full subcategory of $\CC$ such that (C1) -- (C5) hold.
  Let $\AAA$ be a full subcategory of $\CCfin$ and let $F \in \Ob(\CC)$ be universal and locally finite
  for $\AAA$. If $\AAA$ has finite Ramsey degrees then for every $A \in \Ob(\AAA)$
  there exists an essential coloring $\gamma_A : \hom(A, F) \to t_\AAA(A)$.
\end{PROP}

\section{Ramsey property and extreme amenability}
\label{akpt.sec.RPEA}

Let $G$ be a topological group. Its \emph{action} on $X$ is a mapping $X \times G \to X : (x, g) \mapsto x^g$
such that $x^1 = x$ and $(x^g)^f = x^{gf}$.
A \emph{$G$-flow} is a continuous action of a topological group $G$
on a topological space $X$. A \emph{subflow} of a $G$-flow $X \times G \to X$
is a restriction $Y \times G \to Y$ of the above action to a closed subspace $Y \subseteq X$.
A $G$-flow $X \times G \to X$ is \emph{minimal} if it has no proper closed subflows.
A minimal $G$-flow $u : X \times G \to X$ is \emph{universal minimal $G$-flow}
if every compact minimal $G$-flow $Z \times G \to Z$ is a factor of~$u$.
It is a well-known fact that for a compact Hausdorff space $X$ there is, up to isomorphism of $G$-flows,
a unique universal minimal $G$-flow, usually denoted by $G \curvearrowright M(G)$.

A topological group $G$ is \emph{extremely amenable}
if every $G$-flow $X \times G \to X$ on a compact Hausdroff space $X$ has a joint fixed point,
that is, there is an $x_0 \in X$ such that $x_0^g = x_0$
for all $g \in G$. Since $\Sym(X)$, the group of all permutations on $X$,
carries naturally the topology of pointwise convergence, permutation groups can
be thought of as topological groups. For example, it was shown in~\cite{Pestov-1998} that $\Aut(\QQ, \Boxed<)$
is extremely amenable while $\Sym(X)$ is not in case $X$ is a countably infinite set.

We shall use the following lemma from \cite{KPT}:

\begin{LEM}\label{akpt.lem.kpt1} \cite[Lemma 4.1]{KPT}
  Let $G$ be a topological group. A $G$-flow $X$ has a joint fixed point
  if and only if for every $n \in \NN$, every continuous $f : X \to \RR^n$, every $\epsilon > 0$ and every
  finite $H \subseteq G$ there is an $x \in X$ such that
  $
    (\forall h \in H) \; \| f(x) - f(x^h) \| \le \epsilon
  $,
  where $\| \cdot \|$ denotes the euclidean norm in $\RR^n$.
\end{LEM}

In the remainder of the section we fix a locally small category $\CC$ and its full subcategory $\CCfin$
satisfying (C1)--(C5).

\begin{LEM}\label{akpt.lem.top-AutF}
  Let $F$ be a locally finite object. For $A \in \Ob(\CCfin)$ such that $A \toCC F$ and $e_1, e_2 \in \hom(A, F)$ let
  $
    N_F(e_1, e_2) = \{ f \in \Aut(F) : f \cdot e_1 = e_2 \}
  $.
  Then
  $$
    M_F = \{ N_F(e_1, e_2) : A \in \Ob(\CCfin), \; A \toCC F \text{ and } e_1, e_2 \in \hom(A, F) \}.
  $$
  is a base of a topology $\tau_F$ on $\Aut(F)$.
\end{LEM}
\begin{proof}
  Take any $A, B \in \Ob(\CCfin)$, any $e_1, e_2 \in \hom(A, F)$,
  $f_1, f_2 \in \hom(B, F)$ and any $g \in N_F(e_1, e_2) \sec N_F(f_1, f_2)$. Then
  $g \cdot e_1 = e_2$ and $g \cdot f_1 = f_2$. Since $F$ is locally finite
  there are a $D \in \Ob(\CCfin)$, $r \in \hom(D, F)$,
  $p \in \hom(A, D)$ and $q \in \hom(B, D)$ such that $r \cdot p = e_2$ and $r \cdot q = f_2$.
  Moreover, since $g \cdot e_1 = e_2$ and $g \cdot f_1 = f_2$,
  there is an $s \in \hom(D, F)$ such that the diagram below commutes
  \begin{center}
      \begin{tikzcd}
        D \arrow[rr, "r"] \arrow[rrrr, bend left=20, "s"] & & F & & F \arrow[ll, "g"'] \\
        & A \arrow[ul, "p"] \arrow[ur, "e_2" near end] \arrow[urrr, "e_1"' near end] & & B \arrow[ulll, "q" near end] \arrow[ul, "f_2"' near end]
            \arrow[ur, "f_1"']
      \end{tikzcd}
  \end{center}
  In particular, $g \cdot s = r$, so $g \in N_F(s, r)$. Let us show that $N_F(s, r) \subseteq N_F(e_1, e_2) \sec
  N_F(f_1, f_2)$. Take any $h \in N_F(s, r)$. Then
  $h \cdot s = r$, whence $h \cdot s \cdot p  = r \cdot p$, so $h \cdot e_1 = e_2$.
  Analogously $h \cdot f_1 = f_2$, so $h \in N_F(e_1, e_2) \sec N_F(f_1, f_2)$.
\end{proof}

If this construction is performed in $\CC^\op$ we end up with the dual topology on $\Aut(F)$ which denote by $\tau_F^\op$.

\begin{LEM}
  Let $F \in \Ob(\CC)$ be a locally finite object
  whose automorphisms are finitely separated. Then
  $\Aut(F)$ endowed with the topology $\tau_F$ is a Hausdorff topological group.
\end{LEM}
\begin{proof}
  Take any $f, g \in \Aut(F)$ such that $f \ne g$. Then there is an $A \in \Ob(\CCfin)$ and
  an $e \in \hom(A, F)$ such that $f \cdot e \ne g \cdot e$ because the automorphisms
  of $F$ are finitely separated.
  Clearly, $f \in N_F(e, f \cdot e)$ and $g \in N_F(e, g \cdot e)$. To show that
  $N_F(e, f \cdot e) \sec N_F(e, g \cdot e) = \0$, assume that there is an $h \in N_F(e, f \cdot e) \sec N_F(e, g \cdot e)$.
  Then $h \cdot e = f \cdot e$ and $h \cdot e = g \cdot e$ whence $f \cdot e = g \cdot e$ -- contradiction.
  So, $\Aut(F)$ is a Hausdorff topological space.

  It is easy to show that $\mathstrut^{-1}$ is a continuous map as the inverse image of $N_F(e_1, e_2)$ under $\mathstrut^{-1}$
  is $N_F(e_2, e_1)$. Let us show that the composition $\cdot$ is continuous. The inverse image of
  a basic open set $N_F(e_1, e_2)$ under $\cdot$ is
  $$
    Y_{e_1, e_2} = \{(f, g) \in G^2 : f \cdot g \in N_F(e_1, e_2)\} = \{(f, g) \in G^2 : f \cdot g \cdot e_1 = e_2\}.
  $$
  Then an easy calculation shows that for every $(f_0, g_0) \in Y_{e_1, e_2}$ we have
  $$
    (f_0, g_0) \in N_F(g_0 \cdot e_1, e_2) \times N_F(e_1, g_0 \cdot e_1) \subseteq Y_{e_1, e_2}.
  $$
  Therefore, the composition $\cdot$ is continuous.
\end{proof}

Our intention now is to show that for every homogeneous locally finite object $F$ which is universal for $\CCfin$
we have that $\Aut(F)$ endowed with the topology $\tau_F$
is extremely amenable if and only if $\CCfin$ has the Ramsey property.
Let $G = \Aut(F)$. Since $G_{(e)} = N_F(e,e)$ is an open subgroup of $G$, we have that
$
  \{ G_{(e)} : A \in \Ob(\CCfin), \, e \in \hom(A, F) \}
$
is a neighborhood basis of the identity which consists of open subgroups of~$G$.

\begin{LEM}\label{akpt.lem.const}
  Let $F$ be a homogeneous locally finite object in $\CC$ which is universal for $\CCfin$.
  Assume that $\CCfin$ has the Ramsey property and let $G = \Aut(F)$.
  Fix a $k \ge 2$, an $A \in \Ob(\CCfin)$, an $e \in \hom(A, F)$ and a finite $H \subseteq G$.
  Then for every $k$-coloring $\overline f : G \to k$ which is constant on elements of
  $G / G_{(e)}$ there exists a $g \in G$ such that $\overline f$ is constant on $g \cdot H$.
\end{LEM}
\begin{proof}
  Define a mapping $\psi : G \to \hom(A, F)$ by $\psi(g) = g \cdot e$.
  Then $\psi$ is a surjection because $F$ is homogeneous (take any $e' \in \hom(A, F)$;
  the homogeneity of $F$ yields that there is a $g_0 \in G$ such that $g_0 \cdot e = e'$, so
  $\psi(g_0) = g_0 \cdot e = e'$) and the classes of $\ker \psi$ are $G / G_{(e)}$. Since $\overline f$
  is constant on elements of $G / G_{(e)}$, the $k$-coloring $f : \hom(A, F) \to k$ given by
  $f(g \cdot e) = \overline f(g)$ is well defined.

  Let $H = \{h_0, \ldots, h_{n-1}\}$.
  Because $F$ is locally finite there exist $B \in \Ob(\CCfin)$, $p \in \hom(B, F)$ and $m_i \in \hom(A, B)$ such that
  $p \cdot m_i = h_i \cdot e$, $i < n$.
  \begin{center}
    \begin{tikzcd}
      A \arrow[r, "m_i"] \arrow[rr, bend right=20, "h_i \cdot e"'] & B \arrow[r, "p"] & F
    \end{tikzcd}
  \end{center}
  Since $\CCfin$ has the Ramsey property, Lemma~\ref{akpt.lem.ramseyF} yields that
  there is a $w \in \hom(B, F)$ such that $|f(w \cdot \hom(A, B))| = 1$.
  But $\{m_i : i < n\} \subseteq \hom(A, B)$, so
  $|f(\{w \cdot m_i : i < n\})| = 1$.

  Since $F$ is homogeneous, there is a $g \in G$ such that $g \cdot p = w$.
  \begin{center}
    \begin{tikzcd}
      A \arrow[r, "m_i"] \arrow[rr, bend right=20, "h_i \cdot e"'] & B \arrow[r, "p"] \arrow[rr, bend left=30, "w"] & F & F
    \end{tikzcd}
  \end{center}
  Having in mind that $w \cdot m_i = g \cdot p \cdot m_i = g \cdot h_i \cdot e$
  we conclude that $|f(\{g \cdot h_i \cdot e : i < n\})| = 1$,
  whence $|\overline f(\{g \cdot h_i : i < n\})| = 1$. Therefore, $\overline f$ is
  constant on $g \cdot H$.
\end{proof}

\begin{THM}\label{akpt.thm.extr-amen}
  Let $\CC$ be a locally small category and
  let $\CCfin$ be a full subcategory of $\CC$ such that (C1) -- (C5) hold.
  Let $F$ be a homogeneous locally finite object in $\CC$ which is universal for $\CCfin$
  and whose automorphisms are finitely separated.
  Then $\Aut(F)$ endowed with the topology $\tau_F$
  is extremely amenable if and only if $\CCfin$ has the Ramsey property.
\end{THM}
\begin{proof}
  Let $G = \Aut(F)$.

  $(\Rightarrow)$:
  By Lemma~\ref{akpt.lem.ramseyF} it suffices to show that $F \longrightarrow (B)^{A}_k$ for every integer $k \ge 2$ and
  all $A, B \in \Ob(\CCfin)$ such that $A \toCC B$.

  Fix a $k \ge 2$ and $A, B \in \Ob(\CCfin)$ such that $A \toCC B$
  and let $\chi : \hom(A, F) \to k$ be an arbitrary coloring.
  Note that $X = k^{\hom(A, F)}$ is a compact Hausdorff space with respect to the Tychonoff topology (with $k$ discrete).
  Let $G$ act on $X$ by $(\xi, g) \mapsto \xi^g$, where
  $\xi^g(e) = \xi(g^{-1} \cdot e)$, for all $e \in \hom(A, F)$. Clearly, this action is continuous.
  Then $Y = \overline{\{\chi^g : g \in G\}}$ is compact and $G$-invariant. Therefore, there is a $\chi_0 \in Y$
  such that $\chi_0^g = \chi_0$ for all $g \in G$, whence follows that $\chi_0$ is constant
  (take $e_1, e_2 \in \hom(A, F)$; since $F$ is homogeneous there is a $g \in G$ such that
  $g \cdot e_1 = e_2$; then $\chi_0(e_1) = \chi_0(g^{-1} \cdot e_2) = \chi_0^g(e_2) = \chi_0(e_2)$).

  Since $A, B \in \Ob(\CCfin)$ we have that $H = \hom(A, B)$ is finite by~(C3).
  Let $s \in \hom(B, F)$ be a morphism (which exists because $B \in \Ob(\CCfin)$). Since
  $\chi_0 \in Y$, there is a $g \in G$ such that $\chi_0|_{s \cdot H} = \chi^g|_{s \cdot H}$.
  As $\chi_0$ is constant, $\chi^g$ is constant on $s \cdot H$. In other words,
  $\chi^g(s \cdot e_1) = \chi^g(s \cdot e_2)$ for all $e_1, e_2 \in H$, or, equivalently,
  $\chi(g^{-1} \cdot s \cdot e_1) = \chi(g^{-1} \cdot s \cdot e_2)$ for all $e_1, e_2 \in H$.
  Therefore, $|\chi(w \cdot \hom(A, B))| = 1$ for $w = g^{-1} \cdot s \in \hom(B, F)$.

  $(\Leftarrow)$:
  Assume that $\CCfin$ has the Ramsey property. Take any $G$-flow $X$ and let us show that
  it has a joint fixed point.
  By Lemma~\ref{akpt.lem.kpt1} it suffices to show that
  for every $n \in \NN$, every continuous $f : X \to \RR^n$, every $\epsilon > 0$ and every
  finite $H \subseteq G$ there is an $x \in X$ such that
  $
    (\forall h \in H) \; \| f(x) - f(x^h) \| \le \epsilon
  $,
  where $\| \cdot \|$ denotes the euclidean norm in $\RR^n$.

  Take any $n \in \NN$, continuous $f : X \to \RR^n$, $\epsilon > 0$ and
  finite $H \subseteq G$.
  The mapping $\phi : X \times G \to [0, +\infty) : (x, g) \mapsto \| f(x) - f(x^g)\|$ is continuous
  as a composition of continuous mappings,
  so $W = \phi^{-1}\big([0, \epsilon/3)\big)$ is open. Clearly, $X \times \{\id_F\} \subseteq W$.
  Since $X$ is compact and $W$ open, there exists an open $V \subseteq G$ such that
  $X \times \{\id_F\} \subseteq X \times V \subseteq W$. Since
  $\{ G_{(e)} : A \in \Ob(\CCfin), \, e \in \hom(A, F) \}$
  is a neighborhood basis of the identity, we have $G_{(e)} \subseteq V$ for some $e \in \hom(A, F)$
  where $A \in \Ob(\CCfin)$. Therefore, $X \times G_{(e)} \subseteq W$, or, equivalently,
  \begin{equation}\label{akpt.eq.ss}
    (\forall x \in X) (\forall g \in G_{(e)}) \; \| f(x) - f(x^g)\| < \epsilon/3.
  \end{equation}

  Let us partition $f(X) \subseteq \RR^n$ into finitely many sets $A_0, A_1, \ldots, A_{k-1}$
  so that $\diam(A_i) \le \epsilon/3$ for all $i \in k$. Fix an $x_0 \in X$ and for $i \in k$ let
  $$
    W_i = \{ g \in G : f(x_0^g) \in A_i \} \text{ and } V_i = \{g \cdot G_{(e)} : g \in W_i \}.
  $$
  Clearly, $\{ W_i : i \in k \}$ is a partition of $G$ and $\bigcup_{i \in k} V_i = G / G_{(e)}$.
  Note that $\{ V_i : i \in k \}$ is not necessarily a partition of $G / G_{(e)}$ as it may happen that
  $V_i \sec V_j \ne \0$ for some $i \ne j$. Define a coloring $c : G / G_{(e)} \to k$ as follows:
  \begin{align*}
    &c(V_0) = \{0\},\\
    &c(V_1 \setminus V_0) = \{1\},\\
    &c(V_2 \setminus (V_0 \union V_1)) = \{2\},\\
    & \quad \vdots \\
    &c(V_{k-1} \setminus (V_0 \union \ldots \union V_{k-2})) = \{k-1\},
  \end{align*}
  and then extend it to $\overline c : G \to k$ by letting
  $$
    \overline c(g) = c(g \cdot G_{(e)}).
  $$
  Since $\overline c$ is constant on elements of $G / G_{(e)}$, Lemma~\ref{akpt.lem.const} gives us that
  there is a $q \in G$ and a $p \in k$ such that $\overline c \big(q \cdot (H \union \{\id_F\})\big) = \{p\}$.
  Let us show that $x = x_0^q$ is the element of $X$ we are looking for.

  Take any $h \in H$. Then $\overline c(q \cdot h) = \overline c(q) = p$ whence $c(q \cdot h \cdot G_{(e)}) =
  c(q \cdot G_{(e)}) = p$. In other words, $q \cdot h \cdot G_{(e)} \in V_p$ and $q \cdot G_{(e)} \in V_p$,
  so there exist $u, v \in W_p$ such that $q \cdot h \cdot G_{(e)} = u \cdot G_{(e)}$ and
  $q \cdot G_{(e)} = v \cdot G_{(e)}$. Consequently, $q \cdot h \cdot s = u$ and $q \cdot t = v$ for some
  $s, t \in G_{(e)}$. Then, by definition of $W_p$, we have that
  $f(x^{h \cdot s}) = f(x_0^{q \cdot h \cdot s}) = f(x_0^u) \in A_p$ and
  $f(x^{t}) = f(x_0^{q \cdot t}) = f(x_0^v) \in A_p$, so
  \begin{equation}\label{akpt.eq.ss2}
    \|f(x^{t}) - f(x^{h \cdot s})\| \le \diam(A_p) \le \epsilon / 3.
  \end{equation}
  On the other hand \eqref{akpt.eq.ss} implies
  \begin{equation}\label{akpt.eq.ss3}
    \|f(x) - f(x^{t})\| \le \epsilon / 3 \text{\ \ and\ \ } \|f(x^{h}) - f(x^{h \cdot s})\| \le \epsilon / 3.
  \end{equation}
  because $s, t \in G_{(e)}$. Therefore,
  \begin{align*}
    \|f(x) - f(x^{h})\| \le & \|f(x) - f(x^{t})\| + \|f(x^{t}) - f(x^{h \cdot s})\| + \\
                                   & + \|f(x^{h \cdot s}) - f(x^{h})\| \le \epsilon/3 + \epsilon/3 + \epsilon/3
  \end{align*}
  using \eqref{akpt.eq.ss2} and \eqref{akpt.eq.ss3}.
\end{proof}

\noindent
As an immediate consequence of the Duality Principle we now have:

\begin{COR}
  Let $\CC$ be a locally small category and
  let $\CCfin$ be a full subcategory of $\CC$ such that (C1)$^\op$ -- (C5)$^\op$ hold.
  Let $F$ be a projectively homogeneous, projectively locally finite object in $\CC$ which is projectively universal for $\CCfin$
  and whose automorphisms are projectively finitely separated.
  Then $\Aut(F)$ endowed with the dual topology $\tau_F^\op$
  is extremely amenable if and only if $\CCfin$ has the dual Ramsey property.
\end{COR}

\section{Expansions, group actions and the expansion property}
\label{akpt.sec.exp-group-acts}

Recall that all the categories in this paper are locally small, as we have stipulated at the very beginning of the paper.

An \emph{expansion} of a category $\CC$ is a category $\CC^*$ together with a forgetful functor
$U : \CC^* \to \CC$ (that is, a functor which is surjective on objects and injective on hom-sets).
We shall generally follow the convention that $A, B, C, \ldots$ denote objects from $\CC$
while $\calA, \calB, \calC, \ldots$ denote objects from $\CC^*$.
Since $U$ is injective on hom-sets we may safely assume that
$\hom_{\CC^*}(\calA, \calB) \subseteq \hom_\CC(A, B)$ where $A = U(\calA)$, $B = U(\calB)$.
In particular, $\id_\calA = \id_A$ for $A = U(\calA)$. Moreover, it is safe to drop
subscripts $\CC$ and $\CC^*$ in $\hom_\CC(A, B)$ and $\hom_{\CC^*}(\calA, \calB)$, so we shall
simply write $\hom(A, B)$ and $\hom(\calA, \calB)$, respectively.
Let
$
  U^{-1}(A) = \{\calA \in \Ob(\CC^*) : U(\calA) = A \}
$. Note that this is not necessarily a set.

An expansion $U : \CC^* \to \CC$ is \emph{reasonable} (cf.~\cite{KPT}) if
for every $e \in \hom(A, B)$ and every $\calA \in U^{-1}(A)$ there is a $\calB \in U^{-1}(B)$ such that
$e \in \hom(\calA, \calB)$:
\begin{center}
    \begin{tikzcd}
      \calA \arrow[r, "e"] \arrow[d, dashed, mapsto, "U"'] & \calB \arrow[d, dashed, mapsto, "U"] \\
      A \arrow[r, "e"'] & B
    \end{tikzcd}
\end{center}

An expansion $U : \CC^* \to \CC$ has \emph{unique restrictions} if
for every $\calB \in \Ob(\CC^*)$ and every $e \in \hom(A, U(\calB))$ there is a \emph{unique} $\calA \in U^{-1}(A)$
such that $e \in \hom(\calA, \calB)$:
\begin{center}
    \begin{tikzcd}
      \llap{\hbox{$\restr \calB e = \mathstrut$}}\calA \arrow[r, "e"] \arrow[d, dashed, mapsto, "U"'] & \calB \arrow[d, dashed, mapsto, "U"] \\
      A \arrow[r, "e"'] & B
    \end{tikzcd}
\end{center}
We denote this unique $\calA$ by $\restr \calB e$ and refer to it as the \emph{restriction of $\calB$ along~$e$}.

\begin{LEM}\label{akpt.lem.restr}
  Let $U : \CC^* \to \CC$ be an expansion with unique restrictions.
  
  $(a)$ Let $\calA \in U^{-1}(A)$. Then $\restr{\calA}{\id_A} = \calA$. Hence, if $U(f) = \id_A$ then $f = \id_A$.

  $(b)$ $f \in \hom(\calA, \calB)$ if and only if $\calA = \restr \calB f$.

  $(c)$ Let $f \in \hom(A, B)$ and $g \in \hom(B, C)$ and let $\calC \in U^{-1}(C)$. Then $\restr{(\restr{\calC}{g})}{f} = \restr{\calC}{g \cdot f}$.

  $(d)$ Let $A, B \in \Ob(\CC)$ and let
  $f : A \to B$ be an isomorphism in $\CC$. Take any $\calB \in U^{-1}(B)$ and let $\calA = \restr{\calB}{f}$. Then
  $f : \calA \to \calB$ is an isomorphism in $\CC^*$.
\end{LEM}
\begin{proof}
  $(a)$ and $(b)$ follow immediately from the fact that restrictions are unique, while $(c)$ and $(d)$ follow from
  $(a)$ and $(b)$.
\end{proof}

The proofs of the following two lemmas are straightforward and very similar, so we omit the proof of the first lemma.

\begin{LEM}\label{sbrd.lem.disj-union}
  The expansion $U : \CC^* \to \CC$ is an expansion with unique restrictions if and only if
  for all $A \in \Ob(\CC)$ and all $\calB \in \Ob(\CC^*)$ we have that
  $\hom_\CC(A, U(\calB)) = \bigcup_{\calA \in U^{-1}(A)} \hom_{\CC^*}(\calA, \calB)$
  and this is a disjoint union.
\end{LEM}
\begin{proof}
  Analogous to the proof of Lemma~\ref{sbrd.lem.iso-disj-union}.
\end{proof}

\begin{LEM}\label{sbrd.lem.iso-disj-union}
  Let $U : \CC^* \to \CC$ be an expansion with unique restrictions.
  For $A \in \Ob(\CC)$ let $\calA \in U^{-1}(A)$ be arbitrary, and let
  $\{ \calA_i : i \in I \}$ be all the objects in $\CC^*$ isomorphic to $\calA$
  such that $U(\calA_i) = A$, $i \in I$.

  $(a)$
  $\Aut_\CC(A) = \bigcup_{i \in I} \iso_{\CC^*}(\calA_i, \calA)$
  and this is a disjoint union.
  
  $(b)$
  Suppose that $I$ is finite and that $\Aut(A)$ is finite.
  Then $|\Aut_\CC(A)| = |I| \cdot |\Aut_{\CC^*}(\calA)|$.
\end{LEM}
\begin{proof}
  $(a)$
  Take any $f \in \Aut_\CC(A)$. Then $\restr{\calA}{f} = \calA_i$ for some $i \in I$
  (Lemma~\ref{akpt.lem.restr}), so $f \in \iso_{\CC^*}(\calA_i, \calA)$.
  Conversely,
  take any $f \in \iso_{\CC^*}(\calA_i, \calA)$ for some $i \in I$. Then $f : A \to A$ is
  clearly an automorphism of $A$. The union is disjoint because of unique restrictions.
  
  $(b)$
  Since $\Aut_{\CC^*}(\calA) \subseteq \Aut_\CC(A)$ it follows that
  $\Aut_{\CC^*}(\calA)$ is also finite. Note that for each $i \in I$,
  $|\iso_{\CC^*}(\calA_i, \calA)| = |\Aut_{\CC^*}(\calA)|$. The claim now follows from~$(a)$.
\end{proof}

\begin{LEM}\label{akpt.lem.loc-fin-U}
  Let $U : \CC^* \to \CC$ be a reasonable expansion with unique restrictions.
  If $F \in \Ob(\CC)$ is locally finite and $\calF \in U^{-1}(F)$ then $\calF$ is locally finite.
\end{LEM}
\begin{proof}
  Take any $\calA, \calB \in \Ob(\CCfin^*)$ and $e \in \hom(\calA, \calF)$, $f \in \hom(\calB, \calF)$. Let $A = U(\calA)$ and
  $B = U(\calB)$. Note that $\calA = \restr \calF e$ and $\calB = \restr \calF f$.
  
  Since $F$ is locally finite, there exist $C \in \Ob(\CCfin)$, $p \in \hom(A, C)$, $q \in \hom(B, C)$ and $r \in \hom(C, F)$
  such that $r \cdot p = e$ and $r \cdot q = f$. Let $\calC = \restr \calF r$. Let $\calA' = \restr \calC p$.
  Then $p \in \hom(\calA', \calC)$. Moreover, Lemma~\ref{akpt.lem.restr} yields
  $\calA' = \restr{(\restr{\calF}{r})}{p} = \restr{\calF}{r \cdot p} = \restr{\calF}{e} = \calA$, whence
  $p \in \hom(\calA, \calC)$. Analogously, $q \in \hom(\calB, \calC)$.
  
  Assume now that there is an object $\calC' \in \Ob(\CCfin^*)$ together with morphisms
  $p' \in \hom(\calA, \calC')$, $q' \in \hom(\calB, \calC')$ and $r' \in \hom(\calC', \calF)$
  such that $r' \cdot p' = e$ and $r' \cdot q' = f$. Let $C' = U(\calC)$. Then
  $p' \in \hom(A, C')$, $q' \in \hom(B, C')$, $r' \in \hom(C', F)$, and $r' \cdot p' = e$ and $r' \cdot q' = f$.
  Since $F$ is locally finite, there is a morphism $s \in \hom(C, C')$ such that the corresponding
  diagram in $\CC$ commutes.
  Let $\calC'' = \restr{\calC'}{s}$. Since $\calC' = \restr{\calF}{r'}$, Lemma~\ref{akpt.lem.restr} yields
  $\calC'' = \restr{\calC'}{s} = \restr{(\restr{\calF}{r'})}{s} = \restr{\calF}{r' \cdot s} = \restr{\calF}{r} = \calC$,
  so, in particular, $s \in \hom(\calC, \calC')$ and it makes the corresponding diagram in $\CC^*$ commute.
\end{proof}

An expansion $U : \CC^* \to \CC$ is \emph{precompact} (cf.~\cite{vanthe-more}) if
$U^{-1}(A)$ is a set for all $A \in \Ob(\CC)$, and it is a finite set for all $A \in \Ob(\CCfin)$.

\begin{LEM}\label{akpt.lem.age-C*}
  Let $U : \CC^* \to \CC$ be a reasonable precompact expansion with unique restrictions, and
  let $\CCfin^*$ be the full subcategory of $\CC^*$ spanned by the objects
  $$
    \Ob(\CCfin^*) = \UNION \{U^{-1}(A) : A \in \Ob(\CCfin)\}.
  $$
  Then $\CC^*$ and $\CCfin^*$ satisfy (C1)--(C5).
\end{LEM}
\begin{proof}
  (C1) and (C3) are
  trivially satisfied as morphisms in $\CC$ are mono and $\hom(\calA, \calB) \subseteq
  \hom(U(\calA), U(\calB))$ for all $\calA, \calB \in \Ob(\CC^*)$.
  (C2) follows from the fact that $U^{-1}(A)$ is always a set and the fact that $\Ob(\CCfin)$ is a set.
  To show that (C4) holds take any $\calF \in \Ob(\CC^*)$ and let $F = U(\calF)$.
  Because (C4) holds for $\CC$ and $\CCfin$ there is an $A \in \Ob(\CCfin)$ and $e \in \hom(A, F)$.
  Then $\restr \calF e \in U^{-1}(A) \subseteq \Ob(\CCfin^*)$ and $e \in \hom(\restr \calF e, \calF)$ by uniqueness
  of restrictions.
  Finally, to show that (C5) holds take any
  $\calB \in \Ob(\CCfin^*)$ and let $M = \{ \calA \in \Ob(\CCfin^*) : \calA \toCCC \calB \}$ and $B = U(\calB)$; then clearly
  $
    |M| \le \sum \{|U^{-1}(A)| : A \toCC B \},
  $
  which is finite because (C5) holds for $\CC$ and $\CCfin$ and the expansion is precompact.
\end{proof}

Let $U : \CC^* \to \CC$ be a reasonable precompact expansion with unique restrictions and let $F \in \Ob(\CC)$.
For $A \in \Ob(\CCfin)$,  $e \in \hom(A, F)$ and $\calA \in U^{-1}(A)$ let
$$
  N(e, \calA) = \{ \calF \in U^{-1}(F) : e \in \hom(\calA, \calF) \}.
$$
If the expansion $U$ is reasonable then
$N(e, \calA) \ne \0$ for all $e \in \hom(A, F)$ and $\calA \in U^{-1}(A)$.

\begin{LEM}\label{akpt.lem.clopen}
  Let $U : \CC^* \to \CC$ be a reasonable precompact expansion with unique restrictions and let $F \in \Ob(\CC)$.
  Take any $A \in \Ob(\CCfin)$ and $e \in \hom(A, F)$, and let $U^{-1}(A) = \{\calA_0, \ldots, \calA_{n-1}\}$.
  Then $\UNION_{i < n} N(e, \calA_i) = U^{-1}(F)$ and $N(e, \calA_i) \sec N(e, \calA_j) = \0$ whenever $i \ne j$.
\end{LEM}
\begin{proof}
  Let us first show $\UNION_{i < n} N(e, \calA_i) = U^{-1}(F)$.
  The inclusion $\subseteq$ is trivial. As for the inclusion $\supseteq$ take any $\calF \in U^{-1}(F)$. Then $\restr \calF e = \calA_i$
  for some $i$, so $\calF \in N(e, \calA_i)$.
  Assume, now, that there exists some $\calF \in N(e, \calA_i) \sec N(e, \calA_j)$.
  Then $\calA_i = \restr \calF e = \calA_j$ whence $i = j$ because $U$ has unique restrictions.
\end{proof}

Let $U : \CC^* \to \CC$ be an expansion with unique restrictions.
We say that $U$ \emph{separates points} if the following holds:
  for every $F \in \Ob(\CC)$ and all $\calF_1, \calF_2 \in U^{-1}(F)$ such that $\calF_1 \ne \calF_2$
  there exist an $A \in \Ob(\CCfin)$ and an $e \in \hom(A, F)$ such that $\restr{\calF_1}{e} \ne \restr{\calF_2}{e}$.

\begin{PROP}\label{akpt.prop.top-U-1F}
  Let $F \in \Ob(\CC)$ be locally finite
  and let $U : \CC^* \to \CC$ be a reasonable precompact expansion with unique restrictions which separates points.

  $(a)$ $\calS_F = \{ N(e, \calA) : U(\calA) \in \Ob(\CCfin), e \in \hom(U(\calA), F) \}$
  is a base of clopen sets of a topology $\sigma_F$ on $U^{-1}(F)$.
  
  $(b)$ $U^{-1}(F)$ with the topology $\sigma_F$ is a Hausdorff space.
\end{PROP}
\begin{proof}
  $(a)$
  To show that $\calS_F$ is a base of some topology on $U^{-1}(F)$ take any $\calF \in N(e, \calA) \sec N(f, \calB)$:
\begin{center}
    \begin{tikzcd}
      \calA \arrow[r, "e"] \arrow[d, dashed, mapsto, "U"'] & \calF \arrow[d, dashed, mapsto, "U"'] & \calB \arrow[l, "f"'] \arrow[d, dashed, mapsto, "U"] \\
      A \arrow[r, "e"'] & F & B \arrow[l, "f"]
    \end{tikzcd}
\end{center}
  Lemma~\ref{akpt.lem.loc-fin-U} tells us that $\calF$ is locally finite, so there exist $\calC \in \Ob(\CCfin^*)$ and
  $p \in \hom(\calA, \calC)$, $q \in \hom(\calB, \calC)$ and $r \in \hom(\calC, \calF)$ such that
  $r \cdot p = e$ and $r \cdot q = f$. Let $C = U(\calC)$:
\begin{center}
    \begin{tikzcd}
      \calC \arrow[d, dashed, mapsto, "U"'] \arrow[rr, bend left=20, "r"]
      & \calA \arrow[r, "e"'] \arrow[d, dashed, mapsto, "U"'] \arrow[l, "p"]
      & \calF \arrow[d, dashed, mapsto, "U"']
      & \calB \arrow[l, "f"] \arrow[d, dashed, mapsto, "U"] \arrow[lll, bend right=30, "q"']
    \\
      C \arrow[rr, bend right=20, "r"']
      & A \arrow[r, "e"] \arrow[l, "p"']
      & F
      & B \arrow[l, "f"'] \arrow[lll, bend left=30, "q"]
    \end{tikzcd}
\end{center}

  Then clearly $\calF \in N(r, \calC)$. Let us show that $N(r, \calC) \subseteq N(e, \calA) \sec N(f, \calB)$.
  Take any $\calF' \in N(r, \calC)$. Then:
\begin{center}
    \begin{tikzcd}
        \calA \arrow[r, "p"'] \arrow[d, dashed, mapsto, "U"'] \arrow[rr, bend left=20, "e"]
      & \calC \arrow[r, "r"'] \arrow[d, dashed, mapsto, "U"']
      & \calF' \arrow[d, dashed, mapsto, "U"]
    \\
      A \arrow[r, "p"] \arrow[rr, bend right=20, "e"'] & C \arrow[r, "r"] & F
    \end{tikzcd}
\end{center}
  Since $r \cdot p = e$ it follows that $\calF' \in N(e, \calA)$. By the same argument $\calF' \in N(f, \calB)$.
  
  Therefore, $\calS_F$ is a base of a topology on $U^{-1}(F)$. Every basic open set is clopen by Lemma~\ref{akpt.lem.clopen}.
  
  $(b)$
  To show that $(U^{-1}(F), \sigma_F)$
  is a Hausdorff space take any $\calF_1, \calF_2 \in U^{-1}(F)$ such that $\calF_1 \ne \calF_2$.
  Because $U$ separates points
  there exist an $A \in \Ob(\CCfin)$ and an $e \in \hom(A, F)$ such that $\calA_1 = \restr{\calF_1}{e} \ne \restr{\calF_2}{e} = \calA_2$.
  Then $\calF_1 \in N(e, \calA_1)$, $\calF_2 \in N(e, \calA_2)$ and $N(e, \calA_1) \sec N(e, \calA_2) = \0$ (Lemma~\ref{akpt.lem.clopen}).
\end{proof}

Next we show that each reasonable expansion with unique restrictions yields
an action of $\Aut(F)$ on $U^{-1}(F)$ for every $F \in \Ob(\CC)$. We shall show that this action is continuous
with respect to topologies $\tau_F$ on $\Aut(F)$ and $\sigma_F$ on $U^{-1}(F)$
(see Lemma~\ref{akpt.lem.top-AutF} and Proposition~\ref{akpt.prop.top-U-1F}).
We refer to this action as \emph{logical}.

\begin{LEM}
  Let $U : \CC^* \to \CC$ be a reasonable expansion with unique restrictions.
  Let $F \in \Ob(\CC)$ and $g \in \Aut(F)$ be arbitrary but fixed.
  
  $(a)$ Let $\calF, \calF', \calF'' \in U^{-1}(F)$ be such that $g \in \hom(\calF, \calF')$ and $g^{-1} \in \hom(\calF', \calF'')$.
  Then $\calF = \calF''$.
  
  $(b)$ For every $\calF \in U^{-1}(F)$ there exists only one $\calF' \in U^{-1}(F)$ such that $g \in \hom(\calF, \calF')$.
\end{LEM}
\begin{proof}
  $(a)$
  We have the following:
\begin{center}
    \begin{tikzcd}
        \calF \arrow[r, "g"] \arrow[d, dashed, mapsto, "U"']
      & \calF' \arrow[r, "g^{-1}"] \arrow[d, dashed, mapsto, "U"']
      & \calF'' \arrow[d, dashed, mapsto, "U"]
    \\
      F \arrow[r, "g"'] & F \arrow[r, "g^{-1}"'] & F
    \end{tikzcd}
\end{center}
  whence $\calF = \restr{\calF'}{g} =
  \restr{(\restr{\calF''}{g^{-1}})}{g} = \restr{\calF''}{g^{-1}\cdot g} = \restr{\calF''}{\id_F} = \calF''$.

  $(b)$
  Since $U$ is reasonable there exists and $\calF' \in U^{-1}(F)$ such that $g \in \hom(\calF, \calF')$.
  Assume that there is an $\calF'' \in U^{-1}(F)$ such that $g \in \hom(\calF, \calF'')$.
  Then by $(a)$ we have $g^{-1} \in \hom(\calF'', \calF)$:
\begin{center}
    \begin{tikzcd}
        \calF'' \arrow[r, "g^{-1}"] \arrow[d, dashed, mapsto, "U"']
      & \calF \arrow[r, "g"] \arrow[d, dashed, mapsto, "U"']
      & \calF' \arrow[d, dashed, mapsto, "U"]
    \\
      F \arrow[r, "g^{-1}"'] & F \arrow[r, "g"'] & F
    \end{tikzcd}
\end{center}
  As in $(a)$ we now conclude that $\calF' = \calF''$.
\end{proof}

  For $\calF \in \Ob(\CC)$, $g \in \Aut(F)$ and $\calF \in U^{-1}(F)$ let $\calF^g$ denote the unique
  element of $U^{-1}(F)$ satisfying $g^{-1} \in \hom(\calF, \calF^g)$.

\begin{PROP}\label{akpt.prop.cont-act}
  Let $U : \CC^* \to \CC$ be a reasonable precompact expansion with unique restrictions.
  Let $F \in \Ob(\CC)$ be a locally finite object. Then
  $$
    \Aut(F) \times U^{-1}(F) \to U^{-1}(F) : (g, \calF) \mapsto \calF^g
  $$
  is a continuous group action with respect to topologies $\tau_F$ and $\sigma_F$
  (see Lemma~\ref{akpt.lem.top-AutF} and Proposition~\ref{akpt.prop.top-U-1F}).
\end{PROP}
\begin{proof}
  This is clearly a group action. Let us show that it is continuous.
  Let $N(e, \calA)$ be a basic open set in $U^{-1}(F)$, where $e \in \hom(A, F)$ and $A = U(\calA)$.
  Its inverse image is $M = \{(g, \calF) : \calF^g \in N(e, \calA) \}$. Take any $(g_0, \calF_0) \in M$
  and let us show that
  $$
    (g_0, \calF_0) \in N(e, g_0 \cdot e) \times N(g_0 \cdot e, \calA) \subseteq M.
  $$
  Then $g_0 \in N(e, g_0 \cdot e)$ trivially, while $\calF_0 \in N(g_0 \cdot e, \calA)$ because of
\begin{center}
    \begin{tikzcd}
        \calA \arrow[r, "e"] \arrow[d, dashed, mapsto, "U"']
      & \calF_0^{g_0} \arrow[r, "g_0"] \arrow[d, dashed, mapsto, "U"']
      & \calF_0 \arrow[d, dashed, mapsto, "U"]
    \\
      A \arrow[r, "e"'] & F \arrow[r, "g_0"'] & F
    \end{tikzcd}
\end{center}
  (note that the left square states that $(g_0, \calF_0) \in M$, while the right square comes
  from the fact that $g_0^{-1} \in \hom(\calF_0, \calF_0^{g_0})$).
  
  To show that $N(e, g_0 \cdot e) \times N(g_0 \cdot e, \calA) \subseteq M$ take any
  $(h, \calF_1) \in N(e, g_0 \cdot e) \times N(g_0 \cdot e, \calA)$. From $\calF_1 \in N(g_0 \cdot e, \calA)$ we have
\begin{center}
    \begin{tikzcd}
        \calA \arrow[r, "g_0 \cdot e"] \arrow[d, dashed, mapsto, "U"']
      & \calF_1 \arrow[r, "h^{-1}"] \arrow[d, dashed, mapsto, "U"']
      & \calF_1^h \arrow[d, dashed, mapsto, "U"]
    \\
      A \arrow[r, "g_0 \cdot e"] \arrow[rr, bend right=20, "e"'] & F \arrow[r, "h^{-1}"] & F
    \end{tikzcd}
\end{center}
  Since $h \in N(e, g_0 \cdot e)$ we have that $h \cdot e = g_0 \cdot e$, that is, $e = h^{-1} \cdot g_0 \cdot e$.
  Therefore, $\calF_1^h \in N(e, \calA)$.
\end{proof}

Let $\AAA$ and $\AAA^*$ be categories and $U : \AAA^* \to \AAA$ and expansion.
Following \cite{vanthe-more} we say that $U : \AAA^* \to \AAA$ \emph{has the expansion property}
if for every $A \in \Ob(\AAA)$ there exists a $B \in \Ob(\AAA)$ such that $\calA \toAAA \calB$ for all
$\calA, \calB \in \Ob(\AAA^*)$ with $U(\calA) = A$ and $U(\calB) = B$.

\begin{LEM}\label{akpt.lem.EPEQ}
  Let $\AAA$ and $\AAA^*$ be categories such that $\AAA^*$ is directed and all the morphisms in $\AAA$ are mono, and
  let $U : \AAA^* \to \AAA$ be a reasonable expansion with unique restrictions such that $U^{-1}(A)$ is finite
  for all $A \in \Ob(\AAA)$. Then
  $U : \AAA^* \to \AAA$ has the expansion property if and only if
  for every $\calD \in \Ob(\AAA^*)$ there is a $B \in \Ob(\AAA)$
  such that for all $\calB \in \Ob(\AAA^*)$ with $U(\calB) = B$ we have
  $\calD \toAAA \calB$.
\end{LEM}
\begin{proof}
  $(\Rightarrow)$
  Obvious.

  $(\Leftarrow)$
  Take any $A \in \Ob(\AAA)$ and let $U^{-1}(A) = \{\calA_0, \ldots, \calA_{k-1}\}$.
  Because $\AAA^*$ is directed there is a $\calD \in \Ob(\AAA^*)$ such that $\calA_i \toAAA \calD$ for all $i < k$.
  For this $\calD$ there is a $B \in \Ob(\AAA)$ which fulfills the requirement of the lemma.
  Take any $\calA, \calB \in \Ob(\CCfin^*)$ with
  $U(\calA) = A$ and $U(\calB) = B$. Then $\calA = \calA_i$ for some $i < k$, so $\calA = \calA_i \toAAA \calD \toAAA
  \calB$, by the choice of~$B$.
\end{proof}

\begin{LEM}\label{akpt.lem.EP-1}
  Let $U : \CC^* \to \CC$ be a reasonable expansion with unique restrictions.
  Let $F$ be a homogeneous locally finite object of $\CC$ and let $G = \Aut(F)$.
  Then for $\calF, \calF' \in U^{-1}(F)$ we have
  $\calF' \in \overline{\calF^G}$ if and only if
  $\Age(\calF') \subseteq \Age(\calF)$.
\end{LEM}
\begin{proof}
  $(\Rightarrow)$
  Assume that $\calF' \in \overline{\calF^G}$. Let $\calA \in \Age(\calF')$ be arbitrary,
  and let $e \in \hom(\calA, \calF')$. Note that $\calF' \in N(e, \calA)$.
  As every neighborhood of $\calF'$ intersects $\calF^G$,
  it follows that $N(e, \calA) \sec \calF^G \ne \0$. In other words, there is a $g \in G$ such that
  $\calF^g \in N(e, \calA)$. Then
\begin{center}
    \begin{tikzcd}
        \calA \arrow[r, "e"] \arrow[d, dashed, mapsto, "U"']
      & \calF^g \arrow[r, "g"] \arrow[d, dashed, mapsto, "U"']
      & \calF \arrow[d, dashed, mapsto, "U"]
    \\
      A \arrow[r, "e"] & F \arrow[r, "g"] & F
    \end{tikzcd}
\end{center}
  whence follows that $\calA \in \Age(\calF)$ as $g \cdot e \in \hom(\calA, \calF)$.
  
  $(\Leftarrow)$
  Assume that $\Age(\calF') \subseteq \Age(\calF)$ and let us show that every neighborhood of $\calF'$
  intersects $\calF^G$. Let $N(e, \calA)$ be a neighborhood of $\calF'$.
  Then $\calA \in \Age(\calF') \subseteq \Age(\calF)$, so
  there exists a morphism $f \in \hom(\calA, \calF)$. Since $F$ is homogeneous, there is a $g \in \Aut(F)$
  such that $g \cdot f = e$:
\begin{center}
    \begin{tikzcd}
        \calF \arrow[d, dashed, mapsto, "U"']
      & \calA \arrow[d, dashed, mapsto, "U"'] \arrow[r, "e"] \arrow[l, "f"']
      & \calF' \arrow[d, dashed, mapsto, "U"]
    \\
      F \arrow[rr, bend right=20, "g"'] & A \arrow[r, "e"] \arrow[l, "f"'] & F
    \end{tikzcd}
\end{center}
  But then:
\begin{center}
    \begin{tikzcd}
        \calA \arrow[r, "f"'] \arrow[rr, bend left=20, "e"] \arrow[d, dashed, mapsto, "U"']
      & \calF \arrow[r, "g"'] \arrow[d, dashed, mapsto, "U"']
      & \calF^{g^{-1}} \arrow[d, dashed, mapsto, "U"]
    \\
      A \arrow[r, "f"] \arrow[rr, bend right=20, "e"'] & F \arrow[r, "g"] & F
    \end{tikzcd}
\end{center}
  whence follows that $\calF^{g^{-1}} \in N(e, \calA)$.
\end{proof}

\begin{PROP}\label{akpt.prop.EP-2}
  Let $U : \CC^* \to \CC$ be a reasonable expansion with unique restrictions.
  Let $F$ be a locally finite homogeneous object in $\CC$ and assume that $U^{-1}(F)$ is compact
  with respect to the topology $\sigma_F$ (see Proposition~\ref{akpt.prop.top-U-1F}).
  Let $G = \Aut(F)$ and let $\calF \in U^{-1}(F)$ be arbitrary. Then
  $\restr{U}{\Age(\calF)} : \Age(\calF) \to \Age(F)$ has the expansion property if and only if
  $\Age(\calF) \subseteq \Age(\calF')$ for all $\calF' \in \overline{\calF^G}$.
\end{PROP}
\begin{proof}
  $(\Rightarrow)$
  Let $\AAA$ be the full subcategory of $\CCfin$ spanned by $\Age(F)$ and $\AAA^*$ the full subcategory
  of $\CCfin^*$ spanned by $\Age(\calF)$, and
  assume that $\restr{U}{\AAA^*} : \AAA^* \to \AAA$ has the expansion property.
  The fact that $F$ is locally finite and
  Lemma~\ref{akpt.lem.loc-fin-U} yield that $\calF$ is locally finite, whence follows that $\AAA^*$ is directed.
  Take any $\calF' \in \overline{\calF^G}$ and let us show that $\Ob(\AAA^*) = \Age(\calF) \subseteq \Age(\calF')$.
  Let $\calA \in \Age(\calF)$ be arbitrary. Because of the expansion property
  and Lemma~\ref{akpt.lem.EPEQ} there is a $B \in \Ob(\AAA)$ such that $\calA \toCCC \calB$
  for all $\calB \in \Ob(\AAA^*)$ with $U(\calB) = B$.
  Take any $f \in \hom(B, F)$ (which exists because $B \in \Ob(\AAA) = \Age(F)$)
  and let $\calB' = \restr{\calF'}{f}$. Then $\calB' \in \Age(\calF')$.
  Since $\calF' \in \overline{\calF^G}$, Lemma~\ref{akpt.lem.EP-1} yields
  $\Age(\calF') \subseteq \Age(\calF)$, so $\calB' \in \Age(\calF)$.
  The choice of $B$ now ensures that $\calA \toCCC \calB'$. Finally,
  $\calA \toCCC \calB' \toCCC \calF'$ (since $\calB' \in \Age(\calF')$),
  whence $\calA \in \Age(\calF')$.

  $(\Leftarrow)$
  Assume that $\Age(\calF) \subseteq \Age(\calF')$ for all $\calF' \in \overline{\calF^G}$. Let
  $\calA \in \Age(\calF)$ be arbitrary and let $A = U(\calA)$. For $e \in \hom(A, F)$ let
  $$
    X_e = \overline{\calF^G} \sec N(e, \calA).
  $$
  Let us show that
  $$
    \overline{\calF^G} = \UNION \{ X_e : e \in \hom(A, F) \}.
  $$
  The inclusion $\supseteq$ is trivial, while the inclusion $\subseteq$ follows from the assumption. Namely,
  if $\calF' \in \overline{\calF^G}$ then $\Age(\calF) \subseteq \Age(\calF')$; so
  $\calA \in \Age(\calF')$, or, equivalently, there is a morphism
  $f \in \hom(\calA, \calF')$, whence $\calF' \in X_f$.

  By construction each $X_e$ is open in $\overline{\calF^G}$.
  Since $\overline{\calF^G}$ is compact (as a closed subspace of the compact space $U^{-1}(F)$), there is a finite sequence
  $e_0, \ldots, e_{k-1} \in \hom(A, F)$ such that
  $$
    \overline{\calF^G} = \UNION \{ X_{e_j} : j < k \}.
  $$

  Since $F$ is locally finite, there exist $B \in \Ob(\CCfin)$ and morphisms $r \in \hom(B, F)$
  and $p_i \in \hom(A, B)$ such that $r \cdot p_i = e_i$, $i < k$.
  Let us show that for every $\calB \in \Ob(\CCfin^*)$ such that $U(\calB) = B$ we have $\calA \toCCC \calB$.

  Take any $\calB \in \Ob(\CCfin^*)$ such that $U(\calB) = B$ and let $s \in \hom(\calB, \calF)$ be any morphism.
  Then, $s \in \hom(B, F)$, so by the homogeneity of $F$ there is a $g \in G$ such that $g \cdot s = r$.
  Since $g \in \hom(\calF, \calF^{g^{-1}})$ we have that
\begin{center}
    \begin{tikzcd}
        \calF^{g^{-1}} \arrow[d, dashed, mapsto, "U"']
      & \calB          \arrow[d, dashed, mapsto, "U"'] \arrow[r, "s"'] \arrow[l, "g \cdot s"]
      & \calF          \arrow[d, dashed, mapsto, "U"]  \arrow[ll, bend right=20, "g"']
    \\
      F & B \arrow[r, "s"] \arrow[l, "r"'] & F \arrow[ll, bend left=20, "g"] 
    \end{tikzcd}
\end{center}
  In particular, $r \in \hom(\calB, \calF^{g^{-1}})$, so $\calB \in \Age(\calF^{g^{-1}})$.
  Now, $\calF^{g^{-1}} \in \overline{\calF^G} = \UNION \{ X_{e_j} : j < k \}$,
  so $\calF^{g^{-1}} \in X_{e_i}$ for some~$i$. Moreover, $r \cdot p_i = e_i$ by the construction of $B$.
  Therefore:
\begin{center}
    \begin{tikzcd}
        \calA          \arrow[d, dashed, mapsto, "U"'] \arrow[r, "e_i"] 
      & \calF^{g^{-1}} \arrow[d, dashed, mapsto, "U"'] 
      & \calB          \arrow[d, dashed, mapsto, "U"]  \arrow[l, "r"']
    \\
      A \arrow[rr, bend right=20, "p_i"'] \arrow[r, "e_i"] & F & B \arrow[l, "r"']
    \end{tikzcd}
\end{center}
  Let $\calA' = \restr{\calB}{p_i}$. Since $\calB = \restr{\calF^{g^{-1}}}{r}$ we have
  $\calA' = \restr{(\restr{\calF^{g^{-1}}}{r})}{p_i} = \restr{\calF^{g^{-1}}}{r \cdot p_i} =
  \restr{\calF^{g^{-1}}}{e_i} = \calA$. Consequently, $p_i \in \hom(\calA, \calB)$. This concludes the proof
  that $\calA \toCCC \calB$.
\end{proof}

An expansion $U : \CC^* \to \CC$ is \emph{projectively reasonable (resp.\ has unique projective restrictions; has the projective expansion property)} if
$U^\op : (\CC^*)^\op \to \CC^\op$ is reasonable (resp.\ has unique restrictions; has the expansion property).

\section{Ramsey expansions}
\label{akpt.sec.ramsey-expansions}

In this section we show that even in this very abstract setting
the existence of finite Ramsey degrees is a necessary and sufficient condition for the Ramsey expansions to exist.
The key technical tool for this result is a certain additive property of Ramsey degrees.
Recall that all the categories are locally small.

\begin{THM}\label{sbrd.thm.small1}
  Let $U : \CC^* \to \CC$ be a reasonable expansion with restrictions and assume that all the morphisms in $\CC$
  are mono. For any $A \in \Ob(\CC)$ we then have:
  $$
    t_{\CC}(A) \le \sum_{\calA \in U^{-1}(A)} t_{\CC^*}(\calA).
  $$
  Consequently, if $U^{-1}(A)$ is finite and
  $t_{\CC^*}(\calA) < \infty$ for all $\calA \in U^{-1}(A)$ then $t_\CC(A) < \infty$.
\end{THM}
\begin{proof}
  If there is an $\calA \in U^{-1}(A)$ with $t_{\CC^*}(\calA) = \infty$ then the inequality is trivially satisfied. The same
  holds if $U^{-1}(A)$ is infinite. Assume, therefore, that $U^{-1}(A) = \{\calA_1, \calA_2, \ldots, \calA_n\}$ and let
  $t_{\CC^*}(\calA_i) = t_i \in \NN$ for each~$i$.
  
  Take any $k \ge 2$ and any $B \in \Ob(\CC)$ such that $A \toCC B$. Let $h \in \hom_\CC(A, B)$ be
  arbitrary. Because the expansion is reasonable there is a $\calB \in \Ob(\CC^*)$ such that
  $h \in \hom_{\CC^*}(\calA_1, \calB)$. Define inductively $\calC_0$, $\calC_1$, \ldots, $\calC_n \in \Ob(\CC^*)$ so that
  $\calC_0 = \calB$ and $\calC_i \longrightarrow (\calC_{i-1})^{\calA_i}_{k, t_i}$. (Note that such a $\calC_i$ exists because
  $t_{\CC^*}(\calA_i) = t_i$.) Let $C_n = U(\calC_n)$ and let us show that
  $$
    C_n \longrightarrow (B)^A_{k, t_1 + \ldots + t_n}.
  $$
  Take any $\chi : \hom_\CC(A, C_n) \to k$. By Lemma~\ref{sbrd.lem.disj-union} we know that
  $$
    \hom_\CC(A, C_n) = \bigcup_{i=1}^n \hom_{\CC^*}(\calA_i, \calC_n),
  $$
  so we can restrict $\chi$ to each $\hom_{\CC^*}(\calA_i, \calC_n)$ to get $n$ colorings
  $$
    \chi_i : \hom_{\CC^*}(\calA_i, \calC_n) \to k, \quad \chi_i(f) = \chi(f), \quad i \in \{1, \ldots, n\}.
  $$
  Let us construct
  $$
    \chi'_i : \hom_{\CC^*}(\calA_i, \calC_i) \to k \quad \text{and} \quad w_i : \calC_{i-1} \to \calC_i, \quad
    i \in \{1, \ldots, n\}
  $$
  inductively as follows. First, put $\chi'_n = \chi_n$. Given $\chi'_i : \hom_{\CC^*}(\calA_i, \calC_i) \to k$, construct $w_i$
  by the Ramsey property: since $\calC_i \longrightarrow (\calC_{i-1})^{\calA_i}_{k, t_i}$, there is a $w_i : \calC_{i-1} \to \calC_i$
  such that
  $$
    |\chi'_i(w_i \cdot \hom_{\CC^*}(\calA_i, \calC_{i-1}))| \le t_i.
  $$
  Finally, given $w_i : \calC_{i-1} \to \calC_i$ define $\chi'_{i-1} : \hom_{\CC^*}(\calA_{i-1}, \calC_{i-1}) \to k$ by
  $$
    \chi'_{i-1}(f) = \chi_{i-1}(w_n \cdot \ldots \cdot w_i \cdot f).
  $$
  Let us show that
  $$
    |\chi(w_n \cdot \ldots \cdot w_1 \cdot \hom_\CC(A, B))| \le t_1 + \ldots + t_n.
  $$
  By Lemma~\ref{sbrd.lem.disj-union} we know that $\hom_\CC(A, B) = \bigcup_{i=1}^n \hom_{\CC^*}(\calA_i, \calB)$, so
  \begin{align*}
    |\chi(w_n \cdot \ldots \cdot w_1 \cdot \hom_\CC(A, B))|
    & = |\chi(w_n \cdot \ldots \cdot w_1 \cdot \bigcup_{i=1}^n \hom_{\CC^*}(\calA_i, \calB))| \\
    & = |\chi(\bigcup_{i=1}^n w_n \cdot \ldots \cdot w_1 \cdot \hom_{\CC^*}(\calA_i, \calB))| \\
    & = |\bigcup_{i=1}^n \chi(w_n \cdot \ldots \cdot w_1 \cdot \hom_{\CC^*}(\calA_i, \calB))| \\
    & \le \sum_{i=1}^n |\chi(w_n \cdot \ldots \cdot w_1  \cdot \hom_{\CC^*}(\calA_i, \calB))|.
  \end{align*}
  Since $\calB = \calC_0$ and by the constructions of $w_i$'s we have that
  $$
    w_n \cdot \ldots \cdot w_1 \cdot \hom_{\CC^*}(\calA_i, \calB) \subseteq \hom_{\CC^*}(\calA_i, \calC_n).
  $$
  Therefore,
  \begin{align*}
          |\chi(w_n \cdot \ldots \cdot w_1  \cdot \hom_{\CC^*}(\calA_i, \calB))|
    & = |\chi_i(w_n \cdot \ldots \cdot w_1  \cdot \hom_{\CC^*}(\calA_i, \calC_0))| \\
    & = |\chi'_i(w_i \cdot \ldots \cdot w_1 \cdot \hom_{\CC^*}(\calA_i, \calC_0))| \\
    & \le |\chi'_i(w_i \cdot \hom_{\CC^*}(\calA_i, \calC_{i-1}))| \le t_i.
  \end{align*}
  This completes the proof.
\end{proof}

\begin{LEM}\label{sbrd.lem.sml-xp-main}
  Let $U : \CC^* \to \CC$ be a reasonable expansion with unique restrictions which has the expansion property.
  Assume additionally that all the morphisms in $\CC$ are mono and that $\CC^*$ is a directed category. Let
  $A \in \Ob(\CC)$ be arbitrary, let $\calA_1, \ldots, \calA_n \in U^{-1}(A)$ be distinct and assume that
  $t_i = t_{\CC^*}(\calA_i) \in \NN$, $i \in \{1, \ldots, n\}$. Then
  $t_\CC(A) \ge \sum_{i=1}^n t_i$.
\end{LEM}
\begin{proof}
  Since $t_{\CC^*}(\calA_i) = t_i$, $i \in \{1, \ldots, n\}$, for every $i \in \{1, \ldots, n\}$
  there exists a $k_i \ge 2$ and a $\calB_i \in \Ob(\CC^*)$
  such that for every $\calC \in \Ob(\CC^*)$ one can find a coloring $\chi_i : \hom_{\CC^*}(\calA_i, \calC) \to k_i$ such that for every
  $w \in \hom_{\CC^*}(\calB_i, \calC)$ we have that $|\chi_i(w \cdot \hom_{\CC^*}(\calA_i, \calB_i))| \ge t_i$.
  
  Put $k = k_1 + \ldots + k_n$. Since $\CC^*$ is directed, there is a $\calD \in \Ob(\CC^*)$ such that
  $\calB_i \toCCC \calD$ for each $i$.
  Let $\iota_i \in \hom_{\CC^*}(\calB_i, \calD)$ be arbitrary, $i \in \{1, \ldots, n\}$. By the expansion property,
  for $U(\calD) \in \Ob(\CC)$ there is an $E \in \Ob(\CC)$ such that
  $\calD \toCCC \calE$ for all $\calE \in U^{-1}(E)$.
  
  Now, take any $C \in \Ob(\CC)$ such that $E \toCC C$ and any $\calC \in U^{-1}(C)$.
  For every $i \in \{1, \ldots, n\}$ choose a coloring $\chi_i : \hom_{\CC^*}(\calA_i, \calC) \to k_i$ such that for every
  $u \in \hom_{\CC^*}(\calB_i, \calC)$ we have that
  $$
    |\chi_i(u \cdot \hom_{\CC^*}(\calA_i, \calB_i))| \ge t_i.
  $$
  Construct $\chi : \hom_\CC(A, C) \to k = k_1 + \ldots + k_n$ as follows. Having in mind
  Lemma~\ref{sbrd.lem.disj-union},
  \begin{itemize}
  \item[] for $f \in \hom_{\CC^*}(\calA_1, \calC)$ put $\chi(f) = \chi_1(f)$;
  \item[] for $f \in \hom_{\CC^*}(\calA_2, \calC)$ put $\chi(f) = k_1 + \chi_2(f)$;
  \item[] \ \vdots
  \item[] for $f \in \hom_{\CC^*}(\calA_n, \calC)$ put $\chi(f) = k_1 + \ldots + k_{n-1} + \chi_n(f)$;
  \item[] for all other $f \in \hom_\CC(A, C)$ put $\chi(f) = 0$.
  \end{itemize}

  Let $w \in \hom_\CC(E, C)$ be arbitrary. Because $U : \CC^* \to \CC$ has unique restrictions there exists a
  unique $\calE = \restr{\calC}{w}$. Since $E$ was chosen by the expansion property and
  $U(\calE) = E$ there is a $v : \calD \to \calE$.
  
  \begin{center}
    \begin{tikzcd}[row sep=tiny]
      \calA_n & \calB_n \arrow[rd, "\iota_n"] & &  \\
      \vdots & \vdots & \calD \arrow[r, "v"] & \calE \arrow[ddd, mapsto, dashed, "U"'] \arrow[r, "w"] & \calC \arrow[ddd, mapsto, dashed, "U"] \\
      \calA_1 \arrow[dd, mapsto, dashed, "U"'] &  \calB_1 \arrow[ru, "\iota_1"']\\
      \ & \\
      A & & & E \arrow[r, "w"] & C
    \end{tikzcd}
  \end{center}
  
  Let us show that $|\chi(w \cdot \hom_\CC(A, E))| \ge t_1 + \ldots + t_n$. Note, first, that
  $$
    |\chi(w \cdot \hom_\CC(A, E))| \ge |\chi(\bigcup_{i=1}^n w \cdot v \cdot \iota_i \cdot \hom_{\CC^*}(\calA_i, \calB_i))|
  $$
  because
  \begin{align*}
    w \cdot \hom_\CC(A, E)
    & \supseteq w \cdot \bigcup_{i=1}^n \hom_{\CC^*}(\calA_i, \calE)\\
    & = \bigcup_{i=1}^n w \cdot \hom_{\CC^*}(\calA_i, \calE)
        \supseteq \bigcup_{i=1}^n w \cdot v \cdot \iota_i \cdot \hom_{\CC^*}(\calA_i, \calB_i).
  \end{align*}
  The sets $w \cdot v \cdot \iota_i \cdot \hom_{\CC^*}(\calA_i, \calB_i)$, $i \in \{1, \ldots, n\}$, are pairwise
  disjoint (since $w \cdot v \cdot \iota_i \cdot \hom_{\CC^*}(\calA_i, \calB_i) \subseteq \hom_{\CC^*}(\calA_i, \calC)$)
  and, by construction, on each of these sets $\chi$ takes disjoint sets of values (since
  $\hom_{\CC^*}(\calA_1, \calC) \subseteq \{0, \ldots, k_1-1\}$, 
  $\hom_{\CC^*}(\calA_2, \calC) \subseteq \{k_1, \ldots, k_1+k_2-1\}$, and so on). Therefore,
  $$
    |\chi(\bigcup_{i=1}^n w \cdot v \cdot \iota_i \cdot \hom_{\CC^*}(\calA_i, \calB_i))|
    = \sum_{i=1}^n |\chi(w \cdot v \cdot \iota_i \cdot \hom_{\CC^*}(\calA_i, \calB_i))|.
  $$
  As another consequence of the construction of $\chi$ we have that
  $$
    |\chi(w \cdot v \cdot \iota_i \cdot \hom_{\CC^*}(\calA_i, \calB_i))| =
    |\chi_i(w \cdot v \cdot \iota_i \cdot \hom_{\CC^*}(\calA_i, \calB_i))| \ge t_i
  $$
  for all $i \in \{1, \ldots, n\}$, which concludes the proof of the lemma.
\end{proof}

\begin{THM}\label{sbrd.thm.small2}
  Assume that all the morphisms in $\CC$ are mono and that $\CC^*$ is a directed category.
  Let $U : \CC^* \to \CC$ be a reasonable expansion with unique restrictions which has the expansion property.
  Then for any $A \in \Ob(\CC)$ we have the following:
  $$
    t_{\CC}(A) = \sum_{\calA \in U^{-1}(A)} t_{\CC^*}(\calA).
  $$
  Consequently, $t_{\CC}(A)$ is finite if and only if
  $U^{-1}(A)$ is finite and $t_{\CC^*}(\calA) < \infty$ for all $\calA \in U^{-1}(A)$.
\end{THM}
\begin{proof}
  Clearly, it suffices to show the following three facts:
  \begin{itemize}
  \item[(1)]
    if $U^{-1}(A) = \{\calA_1, \calA_2, \ldots, \calA_n\}$ is finite and
    $t_{\CC^*}(\calA_i) < \infty$ for all $i$, then $t_{\CC}(A) = \sum_{i=1}^n t_{\CC^*}(\calA_i)$;
  \item[(2)]
    if $U^{-1}(A)$ is infinite and $t_{\CC^*}(\calA) < \infty$ for all $\calA \in U^{-1}(A)$
    $t_{\CC}(A) = \infty$; and
  \item[(3)]
    if there exists an $\calA \in U^{-1}(A)$ such that $t_{\CC^*}(\calA) = \infty$ then $t_{\CC}(A) = \infty$.
  \end{itemize}

  (1)
  Assume that $U^{-1}(A) = \{\calA_1, \calA_2, \ldots, \calA_n\}$ is finite and that $t_{\CC^*}(\calA_i) < \infty$ for all $i$.
  We have already seen (Theorem~\ref{sbrd.thm.small1}) that $t_{\CC}(A) \le \sum_{i=1}^n t_{\CC^*}(\calA_i)$, and
  that $t_{\CC}(A) \ge \sum_{i=1}^n t_{\CC^*}(\calA_i)$ (Lemma~\ref{sbrd.lem.sml-xp-main}).

  \bigskip
  
  (2)
  Assume that $U^{-1}(A)$ is infinite and that $t_{\CC^*}(\calA) < \infty$ for all $\calA \in U^{-1}(A)$.
  Let us show that $t_\CC(A) = \infty$ by showing that $t_\CC(A) \ge n$ for every $n \in \NN$.
  Fix an $n \in \NN$ and take $n$ distinct $\calA_1, \ldots, \calA_n \in U^{-1}(A)$.
  Then, by Lemma~\ref{sbrd.lem.sml-xp-main}, $t_{\CC}(A) \ge \sum_{i=1}^n t_{\CC^*}(\calA_i) \ge n$.
  
  \bigskip
  
  (3)
  Assume that there is an $\calA \in U^{-1}(A)$ with $t_{\CC^*}(\calA) = \infty$.
  Let us show that $t_\CC(A) = \infty$ by showing that $t_\CC(A) \ge n$ for every $n \in \NN$. Fix an $n \in \NN$.
  The proof is a modification of the proof of Lemma~\ref{sbrd.lem.sml-xp-main}.

  Since $t_{\CC^*}(\calA) = \infty$ there exists a $k \ge 2$ and a $\calB \in \Ob(\CC^*)$
  such that for every $\calC \in \Ob(\CC^*)$ one can find a coloring $\chi' : \hom_{\CC^*}(\calA, \calC) \to k$ such that for every
  $u \in \hom_{\CC^*}(\calB, \calC)$ we have that $|\chi'(u \cdot \hom_{\CC^*}(\calA, \calB))| \ge n$.
  By the expansion property, for $U(\calB) \in \Ob(\CC)$ there is an $E \in \Ob(\CC)$ such that
  $\calB \toCC \calE$ for all $\calE \in U^{-1}(E)$.
  
  Now, take any $C \in \Ob(\CC)$ such that $E \toCC C$ and any $\calC \in U^{-1}(C)$.
  Choose a coloring $\chi' : \hom_{\CC^*}(\calA, \calC) \to k$ such that for every
  $u \in \hom_{\CC^*}(\calB, \calC)$ we have that
  $
    |\chi'(u \cdot \hom_{\CC^*}(\calA, \calB))| \ge n
  $.
  Construct $\chi : \hom_\CC(A, C) \to k$ as follows:
  \begin{itemize}
  \item[] for $f \in \hom_{\CC^*}(\calA, \calC)$ put $\chi(f) = \chi'(f)$;
  \item[] for all other $f \in \hom_\CC(A, C)$ put $\chi(f) = 0$.
  \end{itemize}

  Let $w \in \hom_\CC(E, C)$ be arbitrary. Because $U : \CC^* \to \CC$ has unique restrictions there exists a
  unique $\calE = \restr{\calC}{w}$. Since $E$ was chosen by the expansion property and
  $U(\calE) = E$ there is a $v : \calB \to \calE$.
  \begin{center}
    \begin{tikzcd}[row sep=tiny]
      \calA \arrow[ddd, mapsto, dashed, "U"'] &  \calB \arrow[r, "v"] & \calE \arrow[ddd, mapsto, dashed, "U"'] \arrow[r, "w"] & \calC \arrow[ddd, mapsto, dashed, "U"] \\
      \ & \\
      \ & \\
      A & & E \arrow[r, "w"] & C
    \end{tikzcd}
  \end{center}
  In order to show that $|\chi(w \cdot \hom_\CC(A, E))| \ge n$ note, first, that
  $$
    |\chi(w \cdot \hom_\CC(A, E))| \ge |\chi(w \cdot v \cdot \hom_{\CC^*}(\calA, \calB))|.
  $$
  Since $w \cdot v \cdot \hom_{\CC^*}(\calA, \calB) \subseteq \hom_{\CC^*}(\calA, \calC)$
  we have that
  $$
    \chi(w \cdot v \cdot \hom_{\CC^*}(\calA, \calB)) = \chi'(w \cdot v \cdot \hom_{\CC^*}(\calA, \calB))
  $$
  so, by the choice of $\chi'$,
  $$
    |\chi(w \cdot \hom_\CC(A, E))| \ge |\chi'(w \cdot v \cdot \hom_{\CC^*}(\calA, \calB))| \ge n.
  $$
  This concludes the proof.
\end{proof}

\begin{COR}\label{sbrd.cor.small2-obj}
  Let $U : \CC^* \to \CC$ be a reasonable expansion with unique restrictions which has the expansion property.
  Assume additionally that all the morphisms in $\CC$ are mono and that $\CC^*$ is a directed category.
  Let $A \in \Ob(\CC)$ be such that $\Aut(A)$ is finite.
  
  $(a)$ $t^\sim_{\CC}(A)$ is finite if and only if $U^{-1}(A)$ is finite and $t^\sim_{\CC^*}(\calA) < \infty$ for all $\calA \in U^{-1}(A)$, and in that case
  $$
    t^\sim_{\CC}(A) = \sum_{\calA \in U^{-1}(A)} \frac{|\Aut(\calA)|}{|\Aut(A)|} \cdot t^\sim_{\CC^*}(\calA).
  $$
  
  $(b)$ Assume that $U^{-1}(A)$ is finite and $t^\sim_{\CC^*}(\calA) < \infty$ for all $\calA \in U^{-1}(A)$. Let
  $\calA_1$, \ldots, $\calA_n$ be representatives of isomorphism classes of objects in $U^{-1}(A)$.
  Then
  $$
    t^\sim_{\CC}(A) = \sum_{i=1}^n t^\sim_{\CC^*}(\calA_i).
  $$
\end{COR}
\begin{proof}
  $(a)$
  Since $\Aut(A)$ is finite, Proposition~\ref{rdbas.prop.sml} implies that
  $t^\sim_{\CC}(A)$ is finite if and only if $t_{\CC}(A)$ is finite.
  Moreover, $\Aut(\calA)$ is finite for all $\calA \in U^{-1}(A)$ because $\Aut(\calA) \subseteq \Aut(A)$.
  
  $(\Leftarrow)$
  Assume, first, that $t_\CC(A)$ is not finite. Then $t_{\CC}(A)$ is not finite, so by Theorem~\ref{sbrd.thm.small2},
  $U^{-1}(A)$ is not finite or there is an $\calA \in U^{-1}(A)$ such that $t_{\CC^*}(\calA)$ is not finite.
  The remark at the beginning of the proof then implies that $U^{-1}(A)$ is not finite or there is an $\calA \in U^{-1}(A)$
  such that $t^\sim_{\CC^*}(\calA)$ is not finite.

  $(\Rightarrow)$
  Assume, now, that $t^\sim_\CC(A)$ is finite. Then $t_{\CC}(A)$ is finite, so by Theorem~\ref{sbrd.thm.small2},
  $U^{-1}(A)$ is finite, say $U^{-1}(A) = \{\calA_1, \ldots, \calA_n\}$, and
  $$
    t_{\CC}(A) = \sum_{i=1}^n t_{\CC^*}(\calA_i).
  $$
  By Proposition~\ref{rdbas.prop.sml} we get:
  $$
    |\Aut(A)| \cdot t^\sim_{\CC}(A) = \sum_{i=1}^n |\Aut(\calA_i)| \cdot  t^\sim_{\CC^*}(\calA_i),
  $$
  whence the claim of the corollary follows after dividing by $|\Aut(A)|$.
  
  $(b)$
  By the assumption, $U^{-1}(A) / \Boxed\cong = \{\calA_1 / \Boxed\cong, \ldots, \calA_n / \Boxed\cong \}$. Then
  \begin{align*}
    t^\sim_{\CC}(A)
    &= \sum_{\calA \in U^{-1}(A)} \frac{|\Aut(\calA)|}{|\Aut(A)|} \cdot t^\sim_{\CC^*}(\calA) && \text{by } (a)\\
    &= \sum_{i=1}^n |\calA_i / \Boxed\cong| \cdot \frac{|\Aut(\calA_i)|}{|\Aut(A)|} \cdot t^\sim_{\CC^*}(\calA_i) \\
    &= \sum_{i=1}^n t^\sim_{\CC^*}(\calA_i) && \text{by Lemma~\ref{sbrd.lem.iso-disj-union}}.\square
  \end{align*}
\end{proof}

We are now ready to prove the second main result of the paper:

\begin{THM}\label{akpt.thm.main2} (cf.\ \cite{zucker,vanthe-finramdeg})
  Let $\CC$ be a locally small category and
  let $\CCfin$ be a full subcategory of $\CC$ such that (C1) -- (C5) hold.
  Let $F \in \Ob(\CC)$ be a homogeneous locally finite object, and let $\AAA$ be the full subcategory of
  $\CCfin$ spanned by $\Age(F)$. Then the following are equivalent:
  
  (1) $\AAA$ has finite Ramsey degrees.
  
  (2) There is a reasonable precompact expansion with unique restrictions $U : \CC^* \to \CC$ and
  a full subcategory $\AAA^*$ of $\CC^*$ which is directed, has the Ramsey property and
  $\restr{U}{\AAA^*} : \AAA^* \to \AAA$ has the expansion property.
\end{THM}

The implication $(2) \Rightarrow (1)$ follows immediately from Theorem~\ref{sbrd.thm.small2}.
Therefore, the rest of the section is devoted to proving that $(1) \Rightarrow (2)$.
The construction we are going to present is a generalization of the construction presented in~\cite{vanthe-finramdeg}.

\bigskip

\noindent
\textbf{Standing assumption.} For the remainder of this section
fix a locally small category $\CC$ and its full subcategory $\CCfin$ such that (C1) -- (C5) hold.
Fix a homogeneous locally finite object $F \in \Ob(\CC)$, let $G = \Aut(F)$ and let
$\AAA$ be the full subcategory of $\CCfin$ spanned by $\Age(F)$. Assume that every
$A \in \Ob(\AAA)$ has a finite Ramsey degree $t_\AAA(A)$.

\bigskip

Let $\CC^*$ be the category whose objects are pairs $\calC = (C, \theta)$ where
$C \in \Ob(\CC)$ and $\theta = (\theta_A)_{A \in \Ob(\AAA)}$ is a family of colorings
$$
  \theta_A : \hom(A, C) \to t_\AAA(A)
$$
indexed by the objects of $\AAA$. For such a family $\theta = (\theta_A)_{A \in \Ob(\AAA)}$ and
an $f \in \hom(A, C)$ we let $\theta(f) = \theta_A(f)$.

Morphisms in $\CC^*$ are morphisms from $\CC$ that preserve colorings. More precisely,
$f$ is a morphism from $\calC = (C, \theta)$ to $\calD = (D, \delta)$ in $\CC^*$
if $f \in \hom(C, D)$ and
$$
  \delta(f \cdot e) = \theta(e), \text{ for all } e \in \UNION_{A \in \Ob(\AAA)} \hom(A, F).
$$
Let $U : \CC^* \to \CC$ be the obvious forgetful functor $(C, \theta) \mapsto C$ and $f \mapsto f$,
and let $\CCfin^*$ be the full subcategory of $\CC^*$ spanned by expansions
of objects from $\CCfin$.

\begin{LEM}\label{akpt.lem.OmegaFin}
  The categories $\CC^*$ and $\CCfin^*$ satisfy (C1) -- (C5) and
  $U : \CC^* \to \CC$ is a reasonable precompact expansion with unique restrictions which separates points.
\end{LEM}
\begin{proof}
  The categories $\CC^*$ and $\CCfin^*$ clearly satisfy (C1) and (C3).
  Also, $U$ is clearly a forgetful functor. The fact that $U$ is reasonable and has unique restrictions
  follows immediately from the way morphisms in $\CC^*$ are defined, while (C4) for $\CC^*$ and $\CCfin^*$
  follows from existence of unique restrictions.

  Let us show that $U$ is a precompact expansion, which at the same time proves that (2) is satisfied.
  Clearly, $U^{-1}(F)$ is a set for all $F \in \Ob(\CC)$.
  Take any $B \in \Ob(\CCfin)$ and any $(B, \beta) \in U^{-1}(B)$.
  According to (C3) and (C5) the set $\hom(A, B)$ is finite for all $A \in \Ob(\AAA)$
  and there are only finitely many objects $A \in \Ob(\AAA)$ such that $\hom(A, B) \ne \0$.
  Therefore, only finitely many of the $\beta_A$'s are nonempty, and for each nonempty $\beta_A$ we have
  only finitely many choices as $\beta_A \in t_\AAA(A)^{\hom(A, B)}$, which is finite.
  
  The fact that $U$ is precompact yields immediately that (C5) holds for $\CC^*$ and $\CCfin^*$.
  
  Finally, let us show that $U$ separates points. Take any $C \in \Ob(\CC)$,
  and any $\calC_1 = (C, \theta_1)$ and $\calC_2 = (C, \theta_2)$ in $\Ob(\CC^*)$.
  Assume that for every $A \in \Ob(\CCfin)$ and every $e \in \hom(A, C)$ we have that
  $\restr{\calC_1}{e} = \restr{\calC_2}{e}$, and let us show that then $\theta_1 = \theta_2$.
  Let $\theta_1(p) = i$ for some $p \in \hom(A, C)$ where $A \in \Ob(\AAA)$. We know that
  $\restr{\calC_1}{p} = \restr{\calC_2}{p}$, so let $\restr{\calC_1}{p} = \restr{\calC_2}{p} = \calA = (A, \alpha)$.
  From $p \in \hom(\calA, \calC_1)$ and $\theta_1(p) = i$ it follows that
  $\alpha(\id_A) = i$. On the other hand, from $p \in \hom(\calA, \calC_2)$ it follows that
  $\theta_2(p) = \alpha(\id_A) = i$.
\end{proof}

\begin{LEM}\label{akpt.lem.U-1compact}
  $U^{-1}(F)$ endowed with the topology $\sigma_F$ (see Proposition~\ref{akpt.prop.top-U-1F})
  is a compact Hausdorff space.
\end{LEM}
\begin{proof}
  We already know (Proposition~\ref{akpt.prop.top-U-1F}~$(b)$) that
  $(U^{-1}(F), \sigma_F)$ is a Hausdorff space, so let us show compactness.
  Let $V = \prod_{A \in \Ob(\AAA)} t_\AAA(A)^{\hom(A, F)}$.
  The mapping $\xi : U^{-1}(F) \to V : (F, \phi) \mapsto \phi$ is clearly bijective.
  The space $V$ together with the product topology is compact but the space $U^{-1}(F)$,
  on the other hand, is topologized by $\sigma_F$. So,
  in order to prove that $(U^{-1}(F), \sigma_F)$ is also compact it suffices to show that
  $\xi$ is open. Let
  $$
    N(e, \calB) = \{ \calF \in U^{-1}(F) : e \in \hom(\calB, \calF) \},
  $$
  be a basic open set in~$\sigma_F$ where $B \in \Ob(\CCfin)$, $e \in \hom(B, F)$ and $\calB = (B, \beta) \in U^{-1}(B)$.
  Since $A, B \in \Ob(\CCfin)$, as in the proof of Lemma~\ref{akpt.lem.OmegaFin} we conclude that
  $\UNION_{A \in \Ob(\AAA)} \hom(A, B)$ is finite. Let
  $
    \UNION_{A \in \Ob(\AAA)} \hom(A, B) = \{f_0, \ldots, f_{n-1}\}
  $
  and let $\beta(f_i) = c_i$, $i < n$. From the way morphisms in $\CC^*$ are defined it follows immediately that
  $$
    N(e, \calB) = \{ (F, \phi) \in U^{-1}(F) : \phi(e \cdot f_i) = c_i, i < n \},
  $$
  whose image under $\xi$ is clearly open in the product topology on~$V$. Therefore, $\xi$ is open.
\end{proof}

Having in mind the Standing assumptions for this section, Proposition~\ref{akpt.prop.glob-fin}
ensures that for every $A \in \Ob(\AAA)$ there is an essential coloring
$\gamma_A : \hom(A, F) \to t_\AAA(A)$. Let $\gamma = (\gamma_A)_{A \in \Ob(\AAA)}$,
and let $\calF_\gamma = (F, \gamma) \in \Ob(\CC^*)$.
As we have seen in Section~\ref{akpt.sec.exp-group-acts}, expansions with unique restrictions
induce group actions in a straightforward manner: for $g \in G$ and $\calF \in U^{-1}(F)$
by $\calF^g$ we denote the unique element of $U^{-1}(F)$ satisfying $g^{-1} \in \hom(\calF, \calF^g)$.
Specifically, if $\calF = (F, \phi)$ and $\calF^g = (F, \phi^g)$ then
$$
  \phi^g(g^{-1} \cdot e) = \phi(e), \text{ for all } A \in \Ob(\AAA) \text{ and } e \in \hom(A, F).
$$
Moreover, this action of $G$ on $U^{-1}(F)$ is continuous (Proposition~\ref{akpt.prop.cont-act}).
Since $U^{-1}(F)$ is compact (Lemma~\ref{akpt.lem.U-1compact}) there is an
$\calF^* = (F, \phi^*) \in \overline{\calF_\gamma^G}$ such that $\overline{(\calF^*)^G}$
is minimal with respect to inclusion. Put
$$
  \AAA^* = \Age(\calF^*).
$$
By the assumption $F$ is locally finite, so Lemma~\ref{akpt.lem.loc-fin-U} ensures that so is $\calF^*$.
This together with $\Ob(\AAA^*) = \Age(\calF^*)$ then implies that $\AAA^*$ is directed.
Moreover, by the choice of $\calF^*$, Lemma~\ref{akpt.lem.EP-1} and Proposition~\ref{akpt.prop.EP-2} yield that
$\restr{U}{\AAA^*} : \AAA^* \to \AAA$ has the expansion property. So,
in order to complete the proof of Theorem~\ref{akpt.thm.main2} we have to show that $\AAA^*$ has
the Ramsey property.

\begin{LEM}\label{akpt.lem.mintexp}
  Each $A \in \Ob(\AAA)$ has at least $t_\AAA(A)$ distinct expansions in $\AAA^*$.
\end{LEM}
\begin{proof}
  Take any $A \in \Ob(\AAA)$ and let $t = t_\AAA(A)$. Then there is a $B \in \Ob(\AAA)$ and a coloring
  $\chi : \hom(A, F) \to t$ such that
  $$
    (\forall w \in \hom(B, F)) |\chi(w \cdot \hom(A, B))| = t.
  $$
  In other words, $\chi^{(w)} : \hom(A, B) \to t : f \mapsto \chi(w \cdot f)$ is surjective for each
  $w \in \hom(B, F)$. Fix an arbitrary $e \in \hom(B, F)$ and let $\calB^* = (B, \beta^*) = \restr{\calF^*}{e}$.
  Then $\calB^* \in \Age(\calF^*)$ and
  \begin{equation}\label{akpt.eq.fin-1}
    \phi^*_A(e \cdot f) = \beta^*_A(f), \text{ for all } f \in \hom(A, B).
  \end{equation}
  Since $\calF^* \in \overline{\calF_\gamma^G}$ it follows by Lemma~\ref{akpt.lem.EP-1} that
  $\Age(\calF^*) \subseteq \Age(\calF_\gamma)$, so $\calB^* \in \Age(\calF_\gamma)$. This means that
  there is a morphism $h \in \hom(\calB^*, \calF_\gamma)$ whence
  \begin{equation}\label{akpt.eq.fin-2}
    \gamma_A(h \cdot f) = \beta^*_A(f), \text{ for all } f \in \hom(A, B).
  \end{equation}
  From \eqref{akpt.eq.fin-1} and \eqref{akpt.eq.fin-2} we, therefore, get
  $$
    \gamma_A(h \cdot f) = \phi^*_A(e \cdot f) , \text{ for all } f \in \hom(A, B),
  $$
  or, equivalently,
  $$
    \gamma_A^{(h)} = (\phi^*_A)^{(e)}.
  $$

  Note that $\gamma_A^{(h)}$ is essential at~$B$ (because $\gamma_A$ is essential)
  so there is a $w \in \hom(B, F)$ such that $\ker \gamma_A^{(h)} \subseteq \ker \chi^{(w)}$.
  Since $\chi^{(w)}$ is a surjective $t$-coloring, it follows that $\gamma_A^{(h)}$ is also a
  surjective $t$-coloring. Hence, $(\phi^*_A)^{(e)}$ is a surjective $t$-coloring.
  Choose $f_0, f_1, \ldots, f_{t-1} \in \hom(A, B)$ so that
  $$
    (\phi^*_A)^{(e)}(f_i) = \phi^*_A(e \cdot f_i) = i, \quad i < t
  $$
  and let $\calA^*_i = (A, \alpha^*_i) = \restr{\calF^*}{e \cdot f_i}$, $i < t$.
  Then for each $i < t$ we have that $\calA^*_i \in \Age(\calF^*) = \Ob(\AAA^*)$ and:
  $$
    \alpha^*_i(\id_A) = \phi^*_A(e \cdot f_i) = i.
  $$
  Therefore, $\calA^*_0$, \ldots, $\calA^*_{t-1}$ are $t$ distinct expansions of $A$ in $\AAA^*$
  (they are distinct because each $\alpha^*_i$ colors $\id_A$ by a different color).
\end{proof}

\begin{LEM}
  $\AAA^*$ has the Ramsey property.
\end{LEM}
\begin{proof}
  Take any $\calA^* \in \Ob(\AAA^*)$ and let $A = U(\calA^*)$. As we have seen,
  $\restr{U}{\AAA^*} : \AAA^* \to \AAA$ is a reasonable expansion with unique restrictions which has the expansion property.
  So, Theorem~\ref{sbrd.thm.small2} applies and
  $$
    t_{\AAA}(A) = \sum_{\calA \in (\restrindex{U}{\AAA^*})^{-1}(A)} t_{\AAA^*}(\calA).
  $$
  Lemma~\ref{akpt.lem.mintexp} implies that $|(\restr{U}{\AAA^*})^{-1}(A)| \ge t_\AAA(A)$.
  The equality above then yields that $t_{\AAA^*}(\calA) = 1$
  for all $\calA \in (\restr{U}{\AAA^*})^{-1}(A)$. In particular, $t_{\AAA^*}(\calA^*) = 1$.
\end{proof}

\noindent
This completes the proof of Theorem~\ref{akpt.thm.main2}.

\bigskip

Finally, as an immediate consequence of the Duality Principle we now have:

\begin{COR}
  Let $\CC$ be a locally small category and
  let $\CCfin$ be a full subcategory of $\CC$ such that (C1)$^\op$ -- (C5)$^\op$ hold.
  Let $F \in \Ob(\CC)$ be a projectively homogeneous, projectively locally finite object, and let $\AAA$ be the full subcategory of
  $\CCfin$ spanned by $\ProjAge(F)$. Then the following are equivalent:
  
  (1) $\AAA$ has finite dual Ramsey degrees.
  
  (2) There is a projectively reasonable precompact expansion with unique projective restrictions $U : \CC^* \to \CC$ and
  a full subcategory $\AAA^*$ of $\CC^*$ which is projectively directed, has the dual Ramsey property and
  $\restr{U}{\AAA^*} : \AAA^* \to \AAA$ has the projective expansion property.
\end{COR}

\section{Acknowledgements}

The author gratefully acknowledges the financial support of the Ministry of Education, Science and Technological Development
of the Republic of Serbia (Grant No.\ 451-03-68/2020-14/200125).

\section{Data Availability Statement}

Data sharing not applicable to this article as no datasets were generated or analysed during the current study.

\end{document}